\documentclass[a4paper]{amsart}
\usepackage{amsthm}
\usepackage{thmtools}
\usepackage{thm-restate}
\usepackage{amsmath}
\usepackage{amssymb}
\usepackage[T1]{fontenc}
\usepackage[colorlinks=true,allcolors=blue]{hyperref}
\usepackage[nameinlink]{cleveref}
\usepackage{tikz}
\usetikzlibrary{shapes.multipart}
\usepackage{calc}
\usepackage{tikz-cd}
\usepackage{xparse}
\usepackage{enumitem}
\usepackage[
  backend=bibtex,
  style=alphabetic,
  natbib=true,
  isbn=false, 
  doi=true, 
  url=false, 
  eprint=false, 
  giveninits=true, 
  backref=true, 
]{biblatex}
\AtEveryBibitem{%
  \clearfield{note}%
}
\renewbibmacro{in:}{} 
\usepackage{doi}

\numberwithin{equation}{section}
\declaretheorem[style=theorem,sibling=equation]{theorem}%
\declaretheorem[style=theorem,sibling=equation]{proposition}%
\declaretheorem[style=theorem,sibling=equation]{lemma}%
\declaretheorem[style=theorem,sibling=equation]{corollary}%
\declaretheorem[style=definition,sibling=equation]{definition}%
\declaretheorem[style=remark,sibling=equation]{remark}%
\DeclareFontFamily{U}{mathb}{\hyphenchar\font45}
\DeclareFontShape{U}{mathb}{m}{n}{
      <5> <6> <7> <8> <9> <10> gen * mathb
      <10.95> mathb10 <12> <14.4> <17.28> <20.74> <24.88> mathb12
      }{}
\DeclareSymbolFont{mathb}{U}{mathb}{m}{n}
\DeclareFontSubstitution{U}{mathb}{m}{n}
\DeclareMathSymbol{\cuadrado}        {2}{mathb}{"05}
\DeclareFontFamily{U}{matha}{\hyphenchar\font45}
\DeclareFontShape{U}{matha}{m}{n}{
      <5> <6> <7> <8> <9> <10> gen * matha
      <10.95> matha10 <12> <14.4> <17.28> <20.74> <24.88> matha12
      }{}
\DeclareSymbolFont{matha}{U}{matha}{m}{n}
\DeclareFontSubstitution{U}{matha}{m}{n}
\DeclareMathSymbol{\circulo}          {2}{matha}{"05}
\newcommand{\operad}[1]{\mathcal{#1}} 
\newcommand{\C}{\operad{C}} 
\newcommand{\A}{\operad{A}} 
\renewcommand{\O}{\operad{O}}
\renewcommand{\P}{\operad{P}}

\newcommand{\F}{\operad{F}}
\newcommand{\J}{\operad{J}}
\newcommand{\BV}{\operad{BV}}
\newcommand{\EBV}{\operad{EBV}}
\newcommand{\CBV}{\operad{CBV}_\infty}
\newcommand{\ECBV}{\operad{ECBV}_\infty}
\newcommand{\Lie}{\mathcal{L}}
\renewcommand{\H}{\mathcal{H}}
\newcommand{\G}{\mathcal{G}}
\newcommand{\E}[1]{\operad{E}(#1)} 
\newcommand{\End}[1]{\operatorname{End}(#1)} 
\NewDocumentCommand{\Diff}{O{} m}{\operatorname{D}_{#1}(#2)} 
\renewcommand{\L}[1]{L_{#1}}
\NewDocumentCommand{\HdR}{O{*} m}{H_{\operatorfont{dR}}^{#1}(#2)} 
\NewDocumentCommand{\CdR}{O{} m}{\Omega^{#1}(#2)} 
\NewDocumentCommand{\SDR}{O{normal} m m o o o}{
  \begin{tikzcd}[ampersand replacement=\&, column sep=#1]
  {#2}\arrow[r, shift left = .5ex, "{#4}"]
  \&
  {#3}\arrow[l, shift left = .5ex, "{#5}"] \arrow[loop right, "{#6}"]
\end{tikzcd}}
\NewDocumentCommand{\sym}{o}{\mathbb{S}_{\IfValueT{#1}{#1}}} 
\NewDocumentCommand{\trees}{o}{\operatorname{Trees}^{\cuadrado,\circulo}_{\IfValueT{#1}{#1}}} 
\newcommand{\Operads}{\operatorname{Op}}
\newcommand{\Ch}{\operatorname{Ch}}
\newcommand{\Fin}{\operatorname{Fin}}
\newcommand{\card}[1]{[#1]}
\newcommand{\SMod}{\operatorname{Mod}^{\sym}}
\newcommand{\colim}{\operatornamewithlimits{colim}}

\newcommand{\shuffle}[2]{\operatorname{Sh}_{#1,#2}}

\newif\ifshowcomments
\showcommentstrue 

\title[Operations on the de Rham cohomology of Poisson\dots]{Operations on the de Rham cohomology of Poisson and Jacobi manifolds}
\author{Ai Guan}
\email[Ai Guan]{aguan@us.es}
\author{Fernando Muro}
\email[Fernando Muro]{fmuro@us.es}
\urladdr[Fernando Muro]{https://personal.us.es/fmuro/}
\address{
  Universidad de Sevilla, 
  Facultad de Matemáticas, 
  Departamento de Álgebra, 
  Calle Tarfia s/n, 
  41012 Sevilla, 
  Spain
}
\thanks{The authors were partially supported by the grant P20\_01109 (JUNTA/FEDER, UE). 
  The second author was also partially supported by the grant PID2020-117971GB-C21
funded by MCIN/AEI/10.13039/501100011033. The second author would like to warmly thank Joana Cirici, Vladimir Dotsenko, Geoffroy Horel, Luis Narváez, and Andrew Tonks for very interesting and useful discussions. Cirici suggested considering generalized Poisson manifolds (\Cref{sec:generalized_poisson}) and Cirici and Horel proposed \Cref{cor:G_BV_formality}.
}

\begin{document}

\begin{abstract}
  We prove that the (homotopy) hypercommutative algebra structure on the de Rham cohomology of a Poisson or Jacobi manifold defined by several authors is (homotopically) trivial, i.e.~it reduces to the underlying (homotopy) commutative algebra structure. We do so by showing that the DG operads which codify the algebraic structure on the de Rham complex of Poisson and Jacobi manifolds, generated by the exterior product and the interior products with the structure polyvector fields, are quasi-isomorphic to the commutative suboperad. Hence, there is no hope to endow the de Rham cohomology of such manifolds with any (higher) structure beyond the well-known (homotopy) commutative algebra structure, which exists for any smooth manifold. We proceed similarly with the commutative $BV_\infty$-algebra structure on the de Rham complex of a generalized Poisson supermanifold. 
\end{abstract}

\maketitle

\tableofcontents

\section{Introduction}

Given a Poisson manifold $(M,\pi)$ with Poisson bivector field $\pi\in\Gamma(\Lambda^2T(M))$, \Citeauthor{koszul_1985_crochet_schoutennijenhuis_cohomologie} defined in \cite{koszul_1985_crochet_schoutennijenhuis_cohomologie} a Lie bracket on the de Rham complex of differential forms $\CdR{M}$, which endows it with a Lie algebra structure of degree $-1$. This bracket is called the \emph{Koszul bracket}. It is compatible with the exterior product, in the sense that they both distribute according to the \emph{Gerstenhaber relation}, which says that $[x,-]$ is a degree $|x|-1$ derivation of the exterior product,
\[[x,yz]=[x,y]z+(-1)^{(|x|-1)|y|}y[x,z].\]

\Citeauthor{koszul_1985_crochet_schoutennijenhuis_cohomologie} noted that his bracket induces the trivial bracket on de Rham cohomology $\HdR{M}$. More recently, \Citeauthor{sharygin_talalaev_2008_lieformality_poisson_manifolds} proved that the Koszul bracket is homotopically trivial at the cochain level. More precisely, they showed in \cite{sharygin_talalaev_2008_lieformality_poisson_manifolds} that $\CdR{M}$ endowed with the Koszul bracket is quasi-isomorphic to $\HdR{M}$ equipped with the abelian Lie bracket, as  degree $-1$ Lie algebras. Hence, the homotopical triviality of the Koszul bracket follows from Lie formality. (Recall that $\CdR{M}$ need not be formal as a commutative algebra.)

The Gerstenhaber relation follows from the fact that the Koszul bracket measures the deviation of the Lie derivative $\Delta=\L{\pi}$ with respect to $\pi$ from being a derivation of the exterior product,
\[[x,y]=\Delta(xy)-\Delta(x)y-(-1)^{|x|}x\Delta(y),\]
and $\Delta$ is a differential operator of degree $-1$ and order $\leq 2$. 
This operator endows the commutative algebra $\CdR{M}$ with the structure of a Batalin--Vilkovisky algebra since, in addition, $\Delta$ satisfies $\Delta^2=0$ and $[\Delta,d]=\Delta d+d\Delta=0$. 

The operator $\Delta$ is trivial in cohomology $\HdR{M}$, just like the Koszul bracket. Moreover, by Cartan's homotopy formula $\Delta=\L{\pi}=[i_\pi,d]=i_\pi d-di_\pi$ is null-homotopic with explicit null-homotopy $i_\pi$, the interior product with respect to $\pi$. This trivialization of $\Delta$ can be used to define a new sequence of operations $m_n$ of arity $n$ and degree $2(2-n)$ on $\HdR{M}$, $n\geq2$, extending the commutative product, which is $m_2$. They assemble to a hypercommutative algebra structure (also known as formal Frobenius manifold structure), see \cite{barannikov_kontsevich_1998_frobenius_manifolds_formality,manin_1999_frobenius_manifolds_quantum,losev_shadrin_2007_zwiebach_invariants_getzler,park_2007_semiclassical_quantum_fields}. There may be other hypercommutative algebra structures on the de Rham cohomology of such manifolds, see e.g.~\cite{cirici_horel_2023_formality_hypercommutative_algebras}, but we will always refer to the previous one. A hypercommutative algebra structure resembles the sequence of operations of an $\A_\infty$- or $\C_\infty$-algebra, but it is totally unrelated. Indeed, a hypercommutative algebra structure is primary, and there is a separate higher order notion of $\H_\infty$-algebra. Actually, $\HdR{M}$ is endowed with a minimal $\H_\infty$-algebra structure extending the aforementioned hypercommutative algebra structure. This evinces the fact that the de Rham complex $\CdR{M}$ is a hypercommutative algebra, see \cite{drummond-cole_vallette_2013_minimal_model_batalin,khoroshkin_markarian_shadrin_2013_hypercommutative_operad_homotopy,dotsenko_shadrin_vallette_2015_rham_cohomology_homotopy}. 

This article was initially motivated by the desire to obtain explicit computations of such hypercommutative and $\H_\infty$-algebra structures on $\HdR{M}$. 

We can better explain the previous plethora of algebraic structures in operadic terms as follows.
Consider the operads for the following kinds of algebras:
\begin{itemize}
  \item $\C$ is the operad for commutative algebras.
  \item $\Lie$ is the operad for Lie algebras.
  \item $\Lie[1]$ is the operadic suspension of $\Lie$. An $\Lie[1]$-algebra structure on a complex $X$ is the same as an $\Lie$-algebra structure on its desuspension $X[-1]$.
  \item $\G$ is the operad for Gerstenhaber algebras.
  \item $\BV$ is the operad for Batalin--Vilkovisky algebras.
  \item $\H$ is the operad for hypercommutative algebras.
\end{itemize}
All these operads are graded operads, i.e.~DG operads with trivial differential. Moreover, all of them, with one exception, are quadratic Koszul, see \cite{loday_vallette_2012_algebraic_operads} and the references therein. The only exception is $\BV$, which is not quadratic, but almost, and it is Koszul in a broader sense \cite{galvez-carrillo_tonks_vallette_2012_homotopy_batalin_vilkovisky}. In addition, we consider the DG operads: 
\begin{itemize}
  \item $\BV/\Delta$, the homotopy quotient of the Batalin--Vilkovisky operad by the suboperad generated by the operator $\Delta$ \cite{drummond-cole_vallette_2013_minimal_model_batalin,khoroshkin_markarian_shadrin_2013_hypercommutative_operad_homotopy}.
  \item $\E{X}$, the endomorphism operad of a complex $X$.
\end{itemize}

We have a commutative diagram in the model category $\Operads$ of DG operads \cite{hinich_1997_homological_algebra_homotopy,hinich_2003_erratum_homological_algebra},
\begin{equation}\label{eq:Poisson_diagram}
  \begin{tikzpicture}[xscale=2.5, yscale=2.3, align=center]
    \node (End) at (0,0) {$\E{\CdR{M}}$};
    \node (BV/D) at (0,1) {$\BV/\Delta$};
    \node (BV) at (0,2) {$\BV$};
    \node (Hypercomm) at (-.75,1.5) {$\H$};
    \node (Gerst) at (0,3) {$\G$};
    \node (Comm) at (-1.5,2) {$\C$};
    \node (Lie) at (1.5,2) {$\Lie[1]$};
    \draw[->] (Gerst) edge node [right] {\scriptsize Koszul bracket \\[-2mm] \scriptsize comes from \\[-2mm] \scriptsize Lie derivative $\L{\pi}$}  (BV);
    \draw[->] (BV) -- (BV/D);
    \draw[->] (BV/D) edge node [fill=white] {\scriptsize interior \\[-2mm] \scriptsize product $i_\pi$} (End);
    \draw[->] (Comm) edge [bend left] node [sloped, above] {\scriptsize inclusion} (Gerst);
    \draw[->] (Lie) edge [bend right] node [sloped, above] {\scriptsize inclusion} (Gerst);
    \draw[->] (Hypercomm) edge node [sloped, above] {$\scriptstyle\sim$} (BV/D);
    \draw[->] (Comm) edge [bend right] node [left=2mm] {\scriptsize exterior \\[-2mm] \scriptsize product} (End);
    \draw[->] (Lie) edge [bend left] node [right=2mm] {\scriptsize Koszul \\[-2mm] \scriptsize bracket} (End);
    \draw[->, bend right=23.5mm] (Gerst) edge node [fill=white, near start] {\scriptsize Gerstenhaber \\[-2mm] \scriptsize relation} (End);
    \draw[->, bend left=23.5mm] (BV) edge node [fill=white, near start] {\scriptsize Lie derivative $\L{\pi}$} (End);
    \draw[-, line width=2mm, white] (Comm) -- (Hypercomm);
    \draw[->] (Comm) edge node [rotate=-32, below=1, fill=white, xshift=-.7mm]  {\scriptsize inclusion} (Hypercomm);
  \end{tikzpicture}
\end{equation}
The arrow with source $\H$ is a quasi-isomorphism by \cite{drummond-cole_vallette_2013_minimal_model_batalin,khoroshkin_markarian_shadrin_2013_hypercommutative_operad_homotopy}. The minimal $\H_\infty$-algebra structure on $\HdR{M}$ is obtained from the hypercommutative algebra structure on $\CdR{M}$ applying the well-known \emph{homotopy transfer theorem} \cite[Theorem 10.3.1]{loday_vallette_2012_algebraic_operads}.

Let $\EBV$ be the DG operad whose algebras are commutative algebras $A$ endowed with a differential operator $i\colon A\to A$ of degree $-2$ and order $\leq 2$ such that the following triple commutator involving $i$ and the cochain differential $d\colon A\to A$ vanishes,
\[[i,[i,d]]=0.\]
The acronym $\EBV$ stands for \emph{exact Batalin--Vilkovisky algebra}, since any such algebra gives rise to a Batalin--Vilkovisky algebra by setting $\Delta=[i,d]$. The canonical example is $A=\CdR{M}$ with $i=i_\pi$ the interior product with respecto to $\pi$. In this case, the vanishing condition follows from the equation $[\pi,\pi]=0$ and the formulas relating the Schouten--Nijenhuis bracket in polyvector fields $\Gamma(\Lambda^*T(M))$ with commutators of Lie differentials and interior products on $\CdR{M}$ \cite{cartier_1994_fundamental_techniques_theory}. 

The null-homotopy for the Batalin--Vilkovisky operator $\Delta$ on $\CdR{M}$ is the interior product $i_\pi$, hence we can easily prove the following result.

\begin{restatable}{proposition}{Poissonfactorization}\label{prop:Poisson_factorization}
  Given a Poisson manifold $(M,\pi)$, the morphism $\BV/\Delta\to\E{\CdR{M}}$ in \eqref{eq:Poisson_diagram} factors as \[\BV/\Delta\longrightarrow\EBV\longrightarrow\E{\CdR{M}}.\]
\end{restatable}

One of our main results establishes that the exact Batalin--Vilkovisky operad is quasi-isomorphic to its commutative suboperad.

\begin{restatable}{theorem}{Poisson}\label{thm:Poisson}
  The DG operad morphism $\C\to\EBV$ obtained as the composition of $\C\to\BV/\Delta$ in \eqref{eq:Poisson_diagram} and the first morphism $\BV/\Delta\longrightarrow\EBV$ in the factorization of \Cref{prop:Poisson_factorization} is a quasi-isomorphism.
\end{restatable}

We derive this and other similar results from a \namecref{thm:ideal} detecting injective quasi-isomorphisms of DG operads (\Cref{thm:ideal}). We obtain the following immediate consequences of \Cref{thm:Poisson}.

\begin{definition}
  The \emph{trivial hypercommutative algebra structure} on a commutative algebra consists of the commutative product $m_2$ and $m_n=0$ for $n>2$.
\end{definition}

The inclusion $\C\subset\H$ is an isomorphism in degree $0$, and it has a necessarily unique retraction $\H\twoheadrightarrow\C$. A trivial hypercommutative algebra is the same as the restriction of scalars along $\H\twoheadrightarrow\C$ of a commutative algebra. 

\begin{restatable}{corollary}{trivialhypercommutativecohomology}\label{cor:trivial_hypercommutative_cohomology}
  For any Poisson manifold $(M,\pi)$, the hypercommutative algebra structure on $\HdR{M}$  defined in \cite{barannikov_kontsevich_1998_frobenius_manifolds_formality,manin_1999_frobenius_manifolds_quantum,losev_shadrin_2007_zwiebach_invariants_getzler,park_2007_semiclassical_quantum_fields} is trivial.
\end{restatable}

\begin{restatable}{corollary}{trivialhypercommutativecohomologystrict}\label{cor:trivial_hypercommutative_cohomology_strict}
  For any Poisson manifold $(M,\pi)$, the hypercommutative algebra structure on $\CdR{M}$ defined in \cite{drummond-cole_vallette_2013_minimal_model_batalin,khoroshkin_markarian_shadrin_2013_hypercommutative_operad_homotopy,dotsenko_shadrin_vallette_2015_rham_cohomology_homotopy} is quasi-isomorphic to the trivial one.
\end{restatable}

At the level of minimal resolutions, we similarly have an inclusion $\C_\infty\subset\H_\infty$ and a retraction $\H_\infty\twoheadrightarrow\C_\infty$, both induced by the previous ones. 

\begin{definition}
  The \emph{trivial $\H_\infty$-algebra} structure on a $\C_\infty$-algebra is the restriction of scalars along $\H_\infty\twoheadrightarrow\C_\infty$ of the $\C_\infty$-algebra structure.
\end{definition}

The de Rham cohomology $\HdR{M}$ carries a minimal $\C_\infty$-algebra structure transferred from the commutative algebra structure on $\CdR{M}$ given by the exterior product. \Cref{cor:trivial_hypercommutative_cohomology_strict} implies the following result.

\begin{restatable}{corollary}{trivialhypercommutativecohomologyinfty}\label{cor:trivial_hypercommutative_cohomology_infty}
  Given a Poisson manifold $(M,\pi)$, the minimal $\H_\infty$-algebra structure on $\HdR{M}$ obtained by homotopy transfer from the hypercommutative algebra structure on $\Omega(M)$ defined in \cite{drummond-cole_vallette_2013_minimal_model_batalin,khoroshkin_markarian_shadrin_2013_hypercommutative_operad_homotopy,dotsenko_shadrin_vallette_2015_rham_cohomology_homotopy} is $\infty$-isomorphic to a trivial one.
\end{restatable}

Recall that a DG algebra $A$ over a graded operad $\P$ is \emph{formal} if $A$ and $H^*(A)$ are quasi-isomorphic as DG algebras over $\P$, where the latter is equipped with the trivial differential. If $\P$ is Koszul, $A$ is formal if and only if the transferred $\P_\infty$-algebra structure on $H^*(A)$ is $\infty$-isomorphic to the plain $\P$-algebra structure induced on cohomology. A manifold $M$ is \emph{formal} if $\CdR{M}$ is formal as a $\C$-algebra, i.e.~if it is quasi-isomorphic to $\HdR{M}$ as DG commutative algebras. When $\CdR{M}$ is equipped with a $\P$-algebra structure, we say that $M$ is \emph{$\P$-formal} if $\CdR{M}$ is formal as a $\P$-algebra.

\begin{restatable}{corollary}{hypercommutativeformality}\label{cor:hypercommutative_formality}
  Given a Poisson manifold $(M,\pi)$, $M$ is $\H$-formal if and only if it is formal.
\end{restatable}

We similarly have triviality notions for algebras and $\infty$-algebras over the operads $\Lie[1]$, $\G$, and $\BV$ in \eqref{eq:Poisson_diagram}. The corresponding versions of \Cref{cor:trivial_hypercommutative_cohomology,cor:trivial_hypercommutative_cohomology_strict,cor:trivial_hypercommutative_cohomology_infty,cor:hypercommutative_formality} hold for exactly the same reasons. In fact, any operadic algebra structure which sits over $\EBV$ will be trivial by \Cref{thm:Poisson}. This rules out all possibilities of defining any homotopically non-trivial algebraic structure on the de Rham complex (or cohomology) of a Poisson manifold out of the exterior product and the interior product with its bivector field.

Homotopical triviality over $\Lie[1]$ is commonly known as being \emph{homotopy abelian}. Formality over $\Lie[1]$ holds for all Poisson manifolds by \cite{sharygin_talalaev_2008_lieformality_poisson_manifolds}, we actually recover this result as a consequence of triviality.

All previous results are proved in \Cref{sec:EBV}.

We have a similar situation for the more general class of Jacobi manifolds $(M,\pi,\eta)$, which are manifolds $M$ equipped with a bivector field $\pi\in\Gamma(\Lambda^2T(M))$ and a vector field $\eta\in \Gamma(T(M))$ such that the following relations hold in $\Gamma(\Lambda^*T(M))$:
\[[\pi,\pi]=2\eta\pi,\qquad [\eta,\pi]=0.\]
The precise statements can be found in \Cref{sec:Jacobi} below.

In the final \namecref{sec:generalized_poisson} (\Cref{sec:generalized_poisson}) we consider generalized Poisson supermanifolds. In this case, the de Rham complex is equipped with the structure of a commutative $BV_\infty$-algebra in the sense of \cite{kravchenko_2000_deformations_batalin_vilkovisky}. We show that it is trivial up to quasi-isomorphism. We also recover one of the main results of \cite{braun_lazarev_2013_homotopy_bv_algebras}, which asserts that it satisfies the degeneration property.

We work over a ground field $k$ of characteristic zero.
In results related to manifolds $k=\mathbb{R}$. We assume the reader has a good background on operads. We take \cite{loday_vallette_2012_algebraic_operads} as a standard reference and use their notation and terminology.
The usual conventions on trees are also succinctly explained in \cite{ginzburg_kapranov_1994_koszul_duality_operads}. Complexes are cochain complexes, i.e.~differentials rise the degree.

\section{Strong deformation retractions}\label{sec:SDRs}

In this section we recall the classical notion of strong deformation retraction of complexes and how they can be tensored in non-symmetric and symmetric ways.

\begin{definition}[{\cite[\S12]{eilenberg_mac_lane_1953_groups_pi}}]\label{def:SDR}
  A \emph{strong deformation retraction (SDR)}, also called \emph{contraction}, is a diagram
  \begin{center}
    \SDR{A}{B}[i][p][h]
  \end{center}
  where $A$ and $B$ are complexes, $i$ and $p$ are cochain maps satisfying
  \[pi=1,\]
  and $h$ is a cochain homotopy (degree $-1$) between $ip$ and the identity map, i.e.~
  \[dh+hd=1-ip.\]
\end{definition}

\begin{remark}
  \Citeauthor{eilenberg_mac_lane_1953_groups_pi} also require the first two of the following \emph{side conditions},
  \begin{align*}
    ph  & =0,  & 
    hi  & =0,  & 
    h^2 & =0.
  \end{align*}
  We can always assume they all hold, 
  see \cite[Remarks p.~199]{gugenheim_1982_perturbation_theory_homology} and \cite[2.1]{lambe_stasheff_1987_applications_perturbation_theory},
  but we do not need them.
\end{remark}

\begin{lemma}[{\cite[Lemma 3.1]{eilenberg_mac_lane_1954_groups_pi_ii}}]\label{def:SDR_tensor_products}
  Given two SDRs
  \begin{center}
    \SDR{A_1}{B_1}[i_1][p_1][h_1],\qquad \SDR{A_2}{B_2}[i_2][p_2][h_2],
  \end{center}
  we have a new SDR
  \begin{center}
    \SDR{A_1\otimes A_2}{B_1\otimes B_2}[i_1\otimes i_2][p_1\otimes p_2][h^\otimes_{1,2}],\qquad $h^\otimes_{1,2}=h_1\otimes 1+i_1p_1\otimes h_2$,
  \end{center}
  that we refer to as their \emph{tensor product}.
\end{lemma}

\begin{remark}\label{rem:SDR_tensor_products}
  The tensor product of SDRs is compatible with the associativity constraint so it can be unambiguously iterated. If we have a sequence of SDRs, $1\leq j\leq n$,
  \begin{center}
    \SDR{A_j}{B_j}[i_j][p_j][h_j],
  \end{center}
  their tensor product is the SDR
  \begin{center}
    \SDR[huge]{A_1\otimes\cdots\otimes A_n}{B_1\otimes\cdots\otimes B_n}[i_1\otimes\cdots\otimes i_n][p_1\otimes\cdots\otimes p_n][h^\otimes_{1,\dots,n}],\qquad
    $\displaystyle h^\otimes_{1,\dots,n}=\sum_{j=1}^ni_1p_1\otimes\cdots\otimes i_{j-1}p_{j-1}\otimes h_j\otimes 1^{\otimes^{n-j}}$,
  \end{center}
  compare \cite[(1) in p.~318]{lambe_1993_resolutions_which_split}.
  
  However, the tensor product is not compatible with the symmetry constraint. In order to explain why and to solve this issue, we introduce some notation. Any permutation $\sigma\in\sym[n]$ gives rise to an isomorphism between $n$-fold tensor products of complexes defined by
  \begin{align*}
    \sigma\colon X_1\otimes\cdots\otimes X_n & \longrightarrow  X_{\sigma^{-1}(1)}\otimes\cdots\otimes X_{\sigma^{-1}(n)},  \\
    x_1\otimes\cdots\otimes x_n              & \;\mapsto\;  \pm x_{\sigma^{-1}(1)}\otimes\cdots\otimes x_{\sigma^{-1}(n)}.
  \end{align*}
  Here, the sign is given by the Koszul rule.
  The composition of these isomorphisms coincides with the product in $\sym[n]$. The symmetry constraint is $(1\;2)\colon X_1\otimes X_2\to X_2\otimes X_1$, and
  \[h^\otimes_{1,2}\neq(1\;2) h^\otimes_{2,1}(1\;2)=h_1\otimes i_2p_2+ 1\otimes h_2.\]
  This map would be an alternative choice for the homotopy in the tensor product of SDRs (\Cref{def:SDR_tensor_products}), it is used in e.g.~\cite[Lemma 2.5]{lambe_stasheff_1987_applications_perturbation_theory}. However, there is a third choice, which is compatible with the symmetry constraint: the average of the previous two. This is generalized by the following straightforward consequence of \Cref{def:SDR_tensor_products}, which uses the fact that, for each $\sigma\in\sym[n]$, the map $\sigma h^\otimes_{\sigma(1),\dots,\sigma(n)}\sigma^{-1}$ is an alternative homotopy for the $n$-fold tensor product.
\end{remark}

\begin{lemma}\label{def:SDR_symmetric_n_tensor_product}
  Let $n\geq 2$. Given $n$ SDRs
  \begin{center}
    \SDR{A_j}{B_j}[i_j][p_j][h_j],\qquad $1\leq j\leq n$,
  \end{center}
  we have a new SDR
  \begin{center}
    \SDR[huge]{A_1\otimes\cdots\otimes A_n}{B_1\otimes\cdots\otimes B_n}[i_1\otimes\cdots\otimes i_n][p_1\otimes\cdots\otimes p_n][h^\otimes_s], $\displaystyle h^\otimes_s=\frac{1}{n!}\sum_{\sigma\in\sym[n]}\sigma h^\otimes_{\sigma(1),\dots,\sigma(n)}\sigma^{-1}$,
  \end{center}
  that we refer to as their \emph{symmetric tensor product}.
\end{lemma}

\begin{remark}\label{rem:SDR_symmetric_n_tensor_product}
  The symmetric $n$-fold tensor product is obviously compatible with the symmetry constraint, i.e.~for any $\tau\in\sym[n]$ the three possible squares in the following diagram commute 
  \begin{center}
    \begin{tikzcd}[ampersand replacement=\&, column sep=25mm]
      A_1\otimes\cdots\otimes A_n\arrow[r, shift left = .5ex, "i_1\otimes\cdots\otimes i_n"]
      \arrow[<-,d,"\tau"']
      \&
      B_1\otimes\cdots\otimes B_n\arrow[l, shift left = .5ex, "p_1\otimes\cdots\otimes p_n"] \arrow[loop right, "h^\otimes_{s}"]
      \arrow[<-,d,"\tau"]\\
      A_{\tau(1)}\otimes\cdots\otimes A_{\tau(n)}\arrow[r, shift left = .5ex, "i_{\tau(1)}\otimes\cdots\otimes i_{\tau(n)}"]
      \&
      B_{\tau(1)}\otimes\cdots\otimes B_{\tau(n)}\arrow[l, shift left = .5ex, "p_{\tau(1)}\otimes\cdots\otimes p_{\tau(n)}"] \arrow[loop right, "h^\otimes_{s}"]
    \end{tikzcd}
  \end{center}
  One of the squares is actually a cylinder (the one involving the homotopies). The cylinder commutes because
  \[\sum_{\sigma\in\sym[n]}\tau\sigma h^\otimes_{\tau\sigma(1),\dots,\tau\sigma(n)}\sigma^{-1}\tau^{-1}=\sum_{\sigma\in\sym[n]}\sigma h^\otimes_{\sigma(1),\dots,\sigma(n)}\sigma^{-1},\]
  since $\sym[n]$ is a group. 
  However, the symmetric tensor product is not compatible with the associativity constraint, so it cannot be iterated. This is why we need a definition for each $n\geq 2$.
\end{remark}

\section{When DG operad extensions are quasi-isomorphisms}\label{sec:quasi-free}

Operads are monoids in the category $\SMod$ of $\sym$-modules with respect to the well-known circle product. The forgetful functor
\[\Operads\longrightarrow\SMod\]
from operads to $\sym$-modules has a left adjoint, the \emph{free operad} functor
\[\F\colon\SMod\longrightarrow\Operads.\]

\begin{definition}\label{def:free_extension}
  A \emph{free extension} of an operad $\O$ is an operad of the form $\P=\O\amalg\F(C)$ for some $\sym$-module $C$. We also say that $\P$ is \emph{free relative to $\O$} or just \emph{relatively free}.
\end{definition}

\begin{remark}\label{rem:free_extension}
  With the notation in \Cref{def:free_extension}, $\P$ is weighted with $\O$ and $C$ concentrated in weights $0$ and $1$, respectively.
  
  In order to describe the underlying $\sym$-module of a relatively free operad, we need the category $\trees$. Objects are trees with root and leaves whose inner vertices have shape $\cuadrado$ or $\circulo$ and such that:
  \begin{itemize}
    \item two vertices with the same shape cannot be adjacent,
    \item $\cuadrado$-shaped vertices cannot be adjacent to the root or to a leaf.
  \end{itemize}
  Morphisms are tree isomorphisms preserving the root, the set of leaves, and the shape of inner vertices. For any such tree $T$, we choose an ordering of the inner vertices and, for each inner vertex, and ordering of the set of incoming edges. 
  
  We define $(\O,C)(T)$ as the tensor product of: 
  \begin{itemize}
    \item $\O(n)$ for each arity $n$ $\circulo$-shaped vertex,
    \item $C(n)$ for each arity $n$ $\cuadrado$-shaped vertex.
  \end{itemize}
  
  Given a tree isomorphism $f\colon T\to T'$, we have an induced isomorphism 
  \[f_*\colon (\O,C)(T)\longrightarrow(\O,C)(T')\]
  defined as follows. The isomorphism $f$ need not preserve the chosen orderings in the sets of inner vertices of $T$ and $T'$, but it induces a permutation. Similarly, for each inner vertex $v\in T$, $f$ induces a permutation between sets of incoming edges of $v\in T$ and $f(v)\in T'$. The isomorphism $f_*$ is defined by: 
  \begin{itemize}
    \item the permutation of the tensor factors induced by the restriction of $f$ to inner vertices, and
    \item for each inner vertex $v\in T$, the automorphism of $\O(n)$ or $C(n)$, according to whether $v$ is $\circulo$- or $\cuadrado$-shaped, induced by the permutation given by the restriction of $f$ to the sets of incoming edges of $v\in T$ and $f(v)\in T'$.
  \end{itemize}
  This defines a functor to the category of complexes,
  \[(\O,C)\colon\trees\longrightarrow\Ch.\]
  
  We consider the functor to the category $\Fin$ of finite sets and bijections 
  \[\ell\colon\trees\longrightarrow\Fin\]
  which sends each tree $T$ to its set $\ell(T)$ of leaves. Consider also the finite set $\card{n}=\{1,\dots,n\}$ and the comma category $\ell\downarrow\card{n}$, $n\geq 0$. An object of $\ell\downarrow\card{n}$ can be regarded as a tree with $n$ leaves labeled by $\card{n}$. Morphisms are isomorphisms in $\trees$ preserving the labels of the leaves.  
  The symmetric group $\sym[n]$ acts on $\ell\downarrow\card{n}$ by post-composition, i.e.~permuting the leaves of a labeled tree. The $\sym$-module underlying $\P$ is given by
  \begin{equation}\label{P(n)_colim}
    \P(n)=\colim_{\ell\downarrow\card{n}}(\O,C).
  \end{equation}
  The weight of $(\O,C)(T)$ is the number of $\cuadrado$-shaped vertices of $T$.
  
  The infinitesimal composition is defined by morphisms
  \[\circ_i\colon (\O,C)(T)\otimes(\O,C)(T')\longrightarrow(\O,C)(T'')\]
  in the following way. The leaves of the trees $T$ and $T'$ are labeled by $\card{n}$ and $\card{n'}$, respectively, and $1\leq i\leq n$. The tree $T''$ is obtained by grafting the root of $T'$ onto the $i$ leaf of $T$ and then, if necessary, contracting the edge $e$ delimited by $\circulo$-shaped vertices possibly created by grating (this happens always except when $T$ or $T'$ is the tree $|$ with no inner vertices). The root of $T''$ is the root of $T$ and the set of leaves of $T''$ is the union of the sets of leaves of $T$ and $T'$, with the $i$ leaf of $T$ removed. The leaves of $T''$ are labeled by $\card{{n+n'-1}}$ according to the following ordering: 
  \begin{itemize}
    \item The leaves coming from $T$ and $T'$ keep their internal order.
    \item All leaves of $T'$ occupy the place of the $i$ leaf of $T$.
  \end{itemize}
  The previous morphism $\circ_i$ is defined as follows. Assume we need to contract the aforementioned edge $e$ in the grafting. Denote by $v$ and $w$ the bottom and top $\circulo$-shaped vertices of $e$, respectively. If $p$ is the arity of $v$, $q$ is the arity of $w$, and $j$ is the place of $e$ in the ordered set of incoming edges of $v$, then we first apply $\circ_j\colon\O(p)\otimes\O(q)\to\O(p+q-1)$ to the tensor factors of $v$ and $w$. Here $\O(p+q-1)$ is regarded as the tensor factor of the vertex $[e]\in T''$ created by contraction. The remaining vertices of $T''$ are the vertices of $T$ except for $v$ and the vertices of $T'$ except for $w$. Therefore, after applying $\circ_j$ as indicated we obtain a tensor product of the same factors as in $(\O,C)(T'')$, possibly in the wrong order (we obtain them in an order determined by the chosen orderings on $T$ and $T'$).  
  We finally proceed like in the definition of $f_*$ above, permuting tensor factors and applying symmetric group actions according to the chosen orderings of inner vertices and incoming edges in $T$, $T'$ and $T''$. If we do not need to contract any edge, then either $T=|$, so $T''=T'$, or $T'=|$, so $T''=T$. In both cases $(\O,C)(|)=k$ concentrated in degree $0$ and $\circ_i$ is the tensor unit constraint.
\end{remark}

\begin{definition}\label{def:SDR_S-modules_and_operads}
  A \emph{strong deformation retraction of $\sym$-modules}
  \begin{center}
    \SDR{A}{B}[i][p][h]
  \end{center}
  is a sequence of SDRs of complexes (\Cref{def:SDR}), $n\geq 0$,
  \begin{center}
    \SDR{A(n)}{B(n)}[i(n)][p(n)][h(n)]
  \end{center}
  such that $i$ and $p$ are are morphisms in $\SMod$
  and each $h(n)\colon B(n)\to B(n)$ is compatible with the action of $\sym[n]$.
  
  A \emph{strong deformation retraction of operads} is an SDR of $\sym$-modules
  \begin{center}
    \SDR{\O}{\P}[i][p][h]
  \end{center}
  where $\O$ and $\P$ are DG operads and $i$ and $p$ are DG operad morphisms.
\end{definition}

\begin{theorem}\label{thm:operad_SDR}
  Let $\P=\O\amalg\F(C)$ be a relatively free DG operad such that $C$ is equipped with an $\sym$-module SDR
  \begin{center}
    \SDR{0}{C}[][][h]
  \end{center}
  Then, there exists a unique SDR of operads
  \begin{center}
    \SDR{\O}{\P}[i][p][h_C]
  \end{center}
  such that: 
  \begin{enumerate}
    \item\label{it:i} $i\colon \O\to \P$ is the inclusion of the first factor of the coproduct.
    \item $p\colon \P\to \O$ is the retraction which maps positive weight elements to zero.
          \item\label{it:hi} $h_Ci=0$.
          \item\label{it:h} The restriction of $h_C$ to $C$ is $h$.
          \item\label{it:h_composition} If $x$ and $y$ are elements of weight $v$ and $w$ in $\P$, respectively, then
          \[h_C(x\circ_iy)=\frac{v}{v+w}h_C(x)\circ_iy+(-1)^{|x|}\frac{w}{v+w}x\circ_ih_C(y).\]
          Here, if $v=w=0$ we understand that $\frac{0}{0}=1$.
  \end{enumerate}
\end{theorem}

\begin{proof}
  Any element of $\P$ is a sum of iterated compositions of elements in $\O$ and $C$. By \eqref{it:hi}, $h_C$ vanishes on elements of $\O$. On elements of $C$, $h_C$ is given by \eqref{it:h}. The definition of $h_C$ on an iterated composition is now determined by 
  \eqref{it:h_composition}, hence uniqueness follows. Existence is a consequence of \Cref{lem:free_extension_SDR} and \Cref{cor:free_extension_SDR_composition} below.
\end{proof}

\begin{remark}
  The fact that, in \Cref{thm:operad_SDR}, the inclusion of the first factor of the coproduct $\O\hookrightarrow\P$ is a quasi-isomorphism comes as no surprise. Indeed, it is a standard acyclic cofibration in the sense of \cite[\S6.4]{hinich_1997_homological_algebra_homotopy} in the model category structure on $\Operads$ where weak equivalences are quasi-isomorphisms and fibrations are surjections, see \cite[Theorem 6.1.1]{hinich_1997_homological_algebra_homotopy} and \cite[Theorem 3.2]{hinich_2003_erratum_homological_algebra}. The relevant fact is the existence of an operad SDR with the claimed properties.
\end{remark}

\begin{remark}
  The convention $\frac{0}{0}=1$ in \Cref{thm:operad_SDR} \eqref{it:h_composition} is irrelevant (we could put any value) since $h_C(x\circ_iy)=h_C(x)=h_C(y)=0$ in case $x$ and $y$ have weight $0$ because this means $x,y\in\O$.
\end{remark}

In the following \namecref{lem:free_extension_SDR}, we construct an operad SDR like the one claimed in \Cref{thm:operad_SDR}, except that we do not check \eqref{it:h_composition}. 

\begin{lemma}\label{lem:free_extension_SDR}
  In the setting of \Cref{def:free_extension} and \Cref{rem:free_extension}, suppose that the $\sym$-module $C$ is equipped with an SDR
  \begin{center}
    \SDR{0}{C}[][][h]
  \end{center}
  If we equip $\O$ with the trivial SDR
  \begin{center}
    \SDR{\O}{\O}[1][1][0]
  \end{center}
  and each $(\O,C)(T)$ with the symmetric tensor product SDR (\Cref{def:SDR_symmetric_n_tensor_product}), then this and \eqref{P(n)_colim} induce an SDR of operads
  \begin{center}
    \SDR{\O}{\P}[i][p][h_C]
  \end{center}
  where $i\colon \O\to \P$ is the inclusion of the first factor of the coproduct,  $p\colon \P\to \O$ is the retraction which maps positive weight elements to zero, and $h_Ci=0$.
\end{lemma}

\begin{proof}
  If $T$ has no $\cuadrado$-shaped inner vertices, then $T$ must be a corolla with a single $\circulo$-shaped inner vertex. In this case $(\O,C)(T)=\O(n)$ and the SDR is the second one in the statement, hence $h_Ci=0$. Otherwise, the SDR is of the form
  \begin{center}
    \SDR{0}{(\O,C)(T)}[][][h_T]
  \end{center}
  We just have to check that, if $f\colon T\to T'$ is an isomorphism in $\ell\downarrow\card{n}$ between trees with at least one $\cuadrado$-shaped inner vertex, then the square
  \begin{center}
    \begin{tikzcd}
      (\O,C)(T)\ar[r,"h_T"]&(\O,C)(T)\\
      (\O,C)(T')\ar[<-,u,"f_*"]\ar[r,"h_{T'}"]&(\O,C)(T')\ar[<-,u,"f_*"']
    \end{tikzcd}
  \end{center}
  commutes. This follows from the definition of $f_*$ and from the facts that each $h(n)$ is compatible with the action of $\sym[n]$ and each $(\O,C)(T)$ is a symmetric tensor product, see \Cref{rem:SDR_symmetric_n_tensor_product}.
\end{proof}

We obtain an explicit formula for the cochain homotopy $h_C$ of \Cref{lem:free_extension_SDR} in the following \namecref{lem:free_extension_SDR_formula}.

\begin{lemma}\label{lem:free_extension_SDR_formula}
  With the notation in \Cref{lem:free_extension_SDR} and its proof, given a labeled tree $T$ in $\ell\downarrow\card{n}$ with some $\cuadrado$-shaped vertex, the homotopy $h_T$ contracting $(\O,C)(T)$ consists of applying $h(p)$ to a single $\cuadrado$-shaped vertex of $T$ at a time ($p$ is the arity of the vertex, and this involves a Koszul sign since $h$ has degree $-1$), summing all of them (there are $m$, where $m$ is the number of $\cuadrado$-shaped vertices of $T$) and dividing by $m$. 
\end{lemma}

\begin{proof}
  The homotopy $h_T$ is a symmetric tensor product of the first two SDRs in the statement of \Cref{lem:free_extension_SDR}. The first one has $0$ on the left hence, with the usual notation, $ip=0$ on it. The second one has identity horizontal arrows and trivial homotopy.   
  Therefore, in the formula for $h_s^\otimes$ in \Cref{def:SDR_symmetric_n_tensor_product}, the only surviving summand in $\sigma h^\otimes_{\sigma(1),\dots,\sigma(n)}\sigma^{-1}$ is the one where $h$ is applied to the first $\cuadrado$-shaped vertex, with respect to the ordering of the inner vertices of $T$ induced by the chosen ordering and the permutation $\sigma$, and nothing is done on the rest of inner vertices. 
  
  A given $\cuadrado$-shaped vertex can be `the first', according to the previous criterion, for different permutations $\sigma$. Here $\sigma\in\sym[m+r]$ were $m$, as declared in the statement, is the number of $\cuadrado$-shaped vertices of $T$, and $r$ is the number of $\circulo$-shaped vertices. Any two such permutations placing first the same $\cuadrado$-shaped vertex 
  differ by the product of three different kinds of permutations:
  \begin{itemize}
    \item A permutation of the $\circulo$-shaped vertices, which we identify with an element of $\sym[r]$.
    \item A permutation of the $\cuadrado$-shaped vertices fixing the given one, which we identify with an element of $\sym[m-1]$.
    \item A shuffle permutation of the $\cuadrado$- and $\circulo$-shaped vertices, which we identify with an element of $\sym[m+r]/\sym[m]\times\sym[r]$.
  \end{itemize}
  The converse is also evident, any two permutations differing in a product of these three kinds of permutations place first the same $\cuadrado$-shaped vertex. The subgroup $G\subset\sym[m+r]$ they generate contains commuting copies of $\sym[m-1]$ and $\sym[r]$ with trivial intersection and $G/\sym[m-1]\times\sym[r]\cong \sym[m+r]/\sym[m]\times\sym[r]$ as sets, so
  \[|G|=\lvert\sym[m-1]\rvert\lvert\sym[r]\rvert\left\lvert\frac{\sym[m+r]}{\sym[m]\times\sym[r]}\right\rvert=(m-1)!r!\frac{(m+r)!}{m!r!}=\frac{(m+r)!}{m}.\]
  Therefore, in the formula of $h_s^\otimes$, the summand consisting of applying the homotopy to a given $\cuadrado$-shaped vertex appears as many times. Since $h_T$ consists of adding all of them up and dividing by $(m+r)!$, the result follows.
\end{proof}

As an immediate consequence, we now derive the formula for the value of $h_C$ on a composite element. This will conclude the proof of \Cref{thm:operad_SDR}.

\begin{corollary}\label{cor:free_extension_SDR_composition}
  The operad SDR in \Cref{lem:free_extension_SDR} satisfies the formula in \Cref{thm:operad_SDR} \eqref{it:h_composition}.
\end{corollary}

We finally consider non-free extensions.

\begin{theorem}\label{thm:ideal}
  In the setting of \Cref{thm:operad_SDR}, if $I\subset \P$ is a DG operadic ideal generated by a set $S$ of weight-homogeneous elements of positive weight such that $h_C(S)\subset I$, then the composite 
  \[\O\stackrel{i}{\hookrightarrow}\P\twoheadrightarrow\P/I\]
  is a quasi-isomorphism $\O\to\P/I$.
\end{theorem}

\begin{proof}
  The inclusion $i\colon\O\hookrightarrow\P$ is a quasi-isomorphism by \Cref{thm:operad_SDR}. 
  We will show that the operad SDR in \Cref{thm:operad_SDR} restricts to an $\sym$-module SDR
  \begin{center}
    \SDR{0}{I}[][][h_C]
  \end{center}
  Hence $I$ has trivial cohomology, so the natural projection $\P\twoheadrightarrow\P/I$ is also a quasi-isomorphism, and the result follows.
  
  By assumption, $I$ is concentrated in positive weight. Hence, with the notation in \Cref{thm:operad_SDR}, $p(I)=0$. Therefore, it is only left to check that $h_C(I)\subset I$. 
  
  Let $S=\amalg_{w\geq 1}S^{(w)}$ be the splitting as a disjoint union of weight-homogeneous elements of positive weight. Since $I$ is generated by $S$ as a DG operadic ideal, it is generated by $S\cup d(S)$ as a graded operadic ideal, i.e.~each element of $I$ can be written as a sum of iterated compositions of an element in $S\cup h(S)$ and some other weight-homogeneous elements of $\P$. We now prove that $h_C$ maps such an iterated composition into $I$, by induction on the composition length.
  
  Let us start with a composition of length $1$, i.e.~$x\in S\cup d(S)$. If $x\in S$ then $h(x)\in I$ by hypothesis. If $x\in d(S^{(w)})$, $w\geq 1$, then $x=d(z)$ for some $z\in S^{(w)}$, hence by the cochain homotopy equation
  \[h_C(x)=h_C(d(z))=z-ip(z)-dh_C(z).\]
  Here $z\in S\subset I$, $p(z)=0$ since $w\geq 1$, and $h_C(z)\in I$ by hypothesis, so $dh_C(z)\in I$ too since $I$ is a DG operadic ideal. As a consequence, $h_C(x)\in I$.
  
  Assume the claim is true for iterated compositions $x\in I$ up to a certain length. We have to prove that $x\circ_iy$ and $y\circ_ix$ also satisfy the claim, for any weight-homogeneous $y\in\P$. Let $v$ and $w$ be the weights of $x$ and $y$, respectively. In particular $v\geq1$ since $x\in I$. 
  We apply \Cref{thm:operad_SDR} \eqref{it:h_composition},
  \[h_C(x\circ_iy)=\frac{v}{v+w}h_C(x)\circ_iy+(-1)^{|x|}\frac{w}{v+w}x\circ_ih_C(y).\]
  Here, $x\circ_ih_C(y)\in I$ since $x\in I$, which is an operadic ideal. Moreover,   
  $h_C(x)\in I$ by induction hypothesis, so $h_C(x)\circ_iy\in I$ too. Hence $h_C(x\circ_iy)\in I$. A similar argument shows that $h_C(y\circ_ix)\in I$. This finishes the proof.
\end{proof}


\section{Poisson manifolds}\label{sec:EBV}

We start with a slight variation of a classical notion. 

\begin{definition}\label{def:differential_operator}
  Given a graded commutative algebra $A$ and some $n\geq 1$, a \emph{differential operator} $D\colon A\to A$ of \emph{order} $\leq n$ is a homogeneous graded vector space morphism satisfying the following formula for any $x_1,\dots,x_{n+1}\in A$:
  \[\sum_{\substack{p+q=n+1\\p\geq 1}}(-1)^p\sum_{\sigma\in\shuffle{p}{q}}(-1)^\epsilon D(x_{\sigma(1)}\cdots x_{\sigma(p)})x_{\sigma(p+1)}\cdots x_{\sigma(n+1)}=0.\]
  Here $\shuffle{p}{q}\subset\sym[p+q]=\sym[n+1]$ denotes the set of \emph{$(p,q)$-shuffles}, i.e.~the permutations $\sigma\in\sym[p+q]$ such that $\sigma(1)<\cdots<\sigma(p)$ and $\sigma(p+1)<\cdots<\sigma(p+q)$, and $\epsilon$ is determined by the Koszul sign rule.
\end{definition}

\begin{remark}
  For $n=1$ the differential operator equation in \Cref{def:differential_operator} is the Leibniz rule.
  
  The notion of higher order differential operators goes back to \cite{grothendieck_1967_elements_geometrie_algebrique}. The original definition has different equivalent formulations, see \cite{akman_ionescu_2008_higher_derived_brackets}. \Cref{def:differential_operator} is equivalent to \cite[\S1(7)]{akman_1997_generalizations_batalinvilkovisky_algebras} instead, which is also used in \cite{kravchenko_2000_deformations_batalin_vilkovisky}. If $A$ is unital, \Cref{def:differential_operator} is equivalent to classical differential operators satisfying $D(1)=0$. In all our explicit examples $A$ is unital and $D(1)=0$ for degree reasons. 
  
  See \cite[\S2]{dotsenko_shadrin_vallette_2013_givental_group_action} for extra properties of classical differential operators.
\end{remark}

We now recall the definition of Batalin--Vilkovisky algebras.

\begin{definition}\label{def:BV_algebra}
  A \emph{Batalin--Vilkovisky algebra} or \emph{$BV$-algebra} is a DG commutative algebra $A$ equipped with a differential operator $\Delta\colon A\to A$ of degree $-1$ and order $\leq 2$ satisfying:
  \begin{enumerate}
    \item\label{it:differential_of_Delta} $[d,\Delta]=d\Delta+\Delta d=0$,
    \item\label{it:Delta_square_zero} $\Delta^2=0$.
  \end{enumerate}
  In \eqref{it:differential_of_Delta}, the bracket is the commutator bracket in the endomorphism graded associative algebra $\End{A}$ of the graded vector space underlying $A$. This endomorphism graded algebra is actually a DG algebra with differential $[d,-]$. 
\end{definition}

\begin{proposition}[\cite{koszul_1985_crochet_schoutennijenhuis_cohomologie}]\label{prop:Poisson_BV}
  The de Rham complex $\CdR{M}$ of a Poisson manifold $(M,\pi)$ is a $BV$-algebra with the exterior differential, the exterior product, and the Lie derivative $\Delta=L_\pi$.
\end{proposition}

Recall the following definition from the introduction.

\begin{definition}\label{def:exact_BV_algebra}
  An \emph{exact Batalin--Vilkovisky algebra} or \emph{exact $BV$-algebra} is a DG commutative algebra $A$ equipped with a differential operator $i\colon A\to A$ of degree $-2$ and order $\leq 2$ satisfying
  \begin{equation}\label{eq:i_commutes_with_its_differential}
    [i,[i,d]]=i^2d-2idi+di^2=0,
  \end{equation}
  i.e.~$idi=\frac{i^2d+di^2}{2}$.
\end{definition}

\begin{proposition}\label{prop:Poisson_EBV}
  The de Rham complex $\CdR{M}$ of a Poisson manifold $(M,\pi)$ is an exact $BV$-algebra with the exterior differential, the exterior product, and $i=i_\pi$.
\end{proposition}

\begin{proof}
  Since $\pi$ is a bivector field, $i_\pi$ is a differential operator of degree $-2$ and order $\leq 2$. Moreover, $[i_\pi,[i_\pi,d]]=[i_\pi,L_\pi]=-[L_\pi,i_\pi]=-i_{[\pi,\pi]}=0$ by the equation $[\pi,\pi]=0$. See \cite[\S3]{cartier_1994_fundamental_techniques_theory}.
\end{proof}

\begin{proposition}\label{lem:EBV_is_BV}
  An exact $BV$-algebra is also a $BV$-algebra with $\Delta=[i,d]=id-di$.
\end{proposition}

\begin{proof}
  The operator $\Delta=[i,d]$ has degree $-1$ because $i$ has degree $-2$ and $d$ has degree $1$. Moreover, $\Delta$ is a differential operator of order $\leq 2$ because $i$ is a differential operator of order $\leq 2$ and $d$ is a differential operator of order $\leq 1$. 
  \Cref{def:BV_algebra} \eqref{it:differential_of_Delta} $[d,\Delta]=[d,[i,d]]=0$ follows from the fact that $[d,-]$ is the differential of the DG algebra $\End{A}$. \Cref{def:BV_algebra} \eqref{it:Delta_square_zero} $\Delta^2=idid-di^2d+didi=0$ is a consequence of \eqref{eq:i_commutes_with_its_differential}. 
\end{proof}

The following notion is due to \cite[\S1.1]{khoroshkin_markarian_shadrin_2013_hypercommutative_operad_homotopy}, see also \cite[Theorem 2.1]{dotsenko_shadrin_vallette_2015_rham_cohomology_homotopy}. We use the sign convention of the second reference.

\begin{definition}\label{def:BV/D_algebra}
  A \emph{trivialized $BV$-algebra} or \emph{$BV$-algebra with Hodge-to-de-Rham degeneration data} is a $BV$-algebra $A$ equipped with graded vector space morphisms $\phi_n\colon A\to A$, $n\geq 1$, of degree $-2n$, such that the following formula holds in the graded algebra $\End{A}[[z]]$ of formal power series on a degree $0$ indeterminate $z$ over the endomorphism graded algebra of the underlying graded vector space of $A$,
  \[e^{\phi(z)}de^{-\phi(z)}=d+\Delta z,\qquad \phi(z)=\sum_{n=1}^\infty\phi_nz^n.\]
  This a compact way of saying that $[\phi_1,d]=\Delta$ and, for each $n\geq 2$,
  \begin{equation}\label{eq:phi_BV}
    \sum_{p=1}^n\frac{1}{p!}\sum_{\substack{q_1+\cdots+q_p=n\\q_1,\dots,q_p\geq 1}}
    [\phi_{q_1},\dots[\phi_{q_{p-1}},[\phi_{q_p},d]]\dots]=0,
  \end{equation}
  see e.g.~\cite[p.~18]{cartier_1994_fundamental_techniques_theory}.
\end{definition}

\begin{theorem}[{\cite[Theorems 2.1 and 3.7]{dotsenko_shadrin_vallette_2015_rham_cohomology_homotopy}}]\label{thm:Poisson_BV/D}
  The de Rham complex $\CdR{M}$ of a Poisson manifold $(M,\pi)$ is a trivialized  $BV$-algebra with the exterior differential, the exterior product,
  \begin{align*}
    \Delta & =L_\pi,  & 
    \phi_1 & =i_\pi,
  \end{align*}
  and $\phi_n=0$ for $n\geq 2$.
\end{theorem}

\begin{proposition}\label{lem:EBV_is_BV/D}
  An exact $BV$-algebra is also a trivialized $BV$-algebra with $\Delta=[i,d]$, $\phi_1=i$, and $\phi_n=0$ for $n\geq 2$.
\end{proposition}

\begin{proof}
  After \Cref{lem:EBV_is_BV}, we only have to check \eqref{eq:phi_BV}. Since the only non-vanishing $\phi_n$ is $\phi_1=i$, this reduces to $[i,d]=\Delta$, which holds by definition, and $[i,\dots,[i,[i,d]]\dots]=0$, which is a consequence of \eqref{eq:i_commutes_with_its_differential}.
\end{proof}

\begin{corollary}\label{cor:Poisson_factors}
  The trivialized $BV$-algebra structure on the de Rham complex $\CdR{M}$ of a Poisson manifold $(M,\pi)$ (\Cref{thm:Poisson_BV/D}) is determined by its exact $BV$-algebra structure (\Cref{prop:Poisson_EBV}) and \Cref{lem:EBV_is_BV/D}.
\end{corollary}

We now consider the operadic viewpoint. $BV$-algebras are algebras over the following operad.

\begin{definition}\label{def:BV_operad}
  The \emph{Batalin--Vilkovisky operad} $\BV$ is the DG operad generated by $\mu\in\BV(2)^0$ and $\Delta\in\BV(1)^{-1}$ with 
  \[\mu\cdot (1\ 2)=\mu\]
  subject to the relations:
  \begin{enumerate}
    \item\label{it:Leibniz} $d(\mu)=0$.
    \item\label{it:associative} $\mu\circ_1\mu=\mu\circ_2\mu$.
    \item\label{it:BV_cycle} $d(\Delta)=0$.
    \item\label{it:BV_square_zero} $\Delta^2=0$.
    \item\label{it:BV_differential operator} $(\mu^2\circ_1\Delta)\cdot[()+(1\ 2)+(1\ 2\ 3)]+\Delta\mu^2=\mu\circ_1(\Delta\mu)\cdot[()+(2\ 3)+(1\ 3\ 2)]$.
  \end{enumerate}
  This is actually a graded operad because $d=0$ but we present it as a DG operad in order to place it in the convenient category. Relations \eqref{it:Leibniz} and \eqref{it:associative} imply that the suboperad generated by $\mu$ is the commutative operad $\C\subset\BV$, \eqref{it:BV_cycle} and \eqref{it:BV_square_zero} correspond to \Cref{def:BV_algebra} \eqref{it:differential_of_Delta} and \eqref{it:Delta_square_zero}, and \eqref{it:BV_differential operator} is \Cref{def:differential_operator} for $n=2$.
\end{definition}

In the following definition we present the DG operad whose algebras are exact Batalin--Vilkovisky algebras.

\begin{definition}\label{def:EBV_operad}
  The \emph{exact Batalin--Vilkovisky operad} $\EBV$ is the DG operad generated by $\mu\in\EBV(2)^0$ and $i\in\EBV(1)^{-2}$ with 
  \[\mu\cdot (1\ 2)=\mu\]
  subject to the relations:
  \begin{enumerate}
    \item $d(\mu)=0$.
    \item $\mu\circ_1\mu=\mu\circ_2\mu$.
          \item\label{it:EBV_commutativity} $id(i)=d(i)i$.
          \item\label{it:EBV_differential operator} $(\mu^2\circ_1i)\cdot[()+(1\ 2)+(1\ 2\ 3)]+i\mu^2=\mu\circ_1(i\mu)\cdot[()+(2\ 3)+(1\ 3\ 2)]$.
  \end{enumerate}
  As in \Cref{def:BV_operad}, the suboperad generated by $\mu$ is the commutative operad $\C\subset\EBV$. Relation \eqref{it:EBV_commutativity} corresponds to \eqref{eq:i_commutes_with_its_differential}, and  \eqref{it:EBV_differential operator} is \Cref{def:differential_operator} for $n=2$.
\end{definition}

The operadic counterpart of \Cref{lem:EBV_is_BV} is the following \namecref{lem:BV_to_EBV}.

\begin{proposition}\label{lem:BV_to_EBV}
  There is a morphism of DG operads $\BV\to\EBV$ given by $\mu\mapsto\mu$ and $\Delta\mapsto -d(i)$.
\end{proposition}

Now it is the turn for the presentation of the operad for trivialized $BV$-algebras.

\begin{definition}\label{def:BV/D_operad}
  The \emph{trivialized Batalin--Vilkovisky operad} $\BV/\Delta$ is the DG operad generated by $\mu\in(\BV/\Delta)(2)^0$, $\Delta\in(\BV/\Delta)(1)^{-1}$, and $\phi_n\in(\BV/\Delta)(1)^{-2n}$, $n\geq 1$, with
  \[\mu\cdot (1\ 2)=\mu\]
  subject to the relations in \Cref{def:BV_operad} together with:
  \[d(\phi_1)=-\Delta,\qquad d(\phi_n)+\sum_{p=2}^n\frac{1}{p!}\sum_{\substack{q_1+\cdots+q_p=n\\q_1,\dots,q_p\geq 1}}
    [\phi_{q_1},\dots[\phi_{q_{p-1}},d(\phi_{q_p})]\dots]=0,\quad n\geq 2.\]
  These new relations correspond to the equations in \Cref{def:BV/D_algebra}. 
  This an honest DG operad with $d\neq 0$. Clearly, $\BV\subset\BV/\Delta$ is a suboperad. The inclusion $\BV\rightarrowtail \BV/\Delta$ is a cofibration in $\Operads$, it is actually a standard cofibration in the sense of \cite[\S6.4]{hinich_1997_homological_algebra_homotopy}.
\end{definition}

The following \namecref{thm:formality_BV/D} is the main result of \cite{khoroshkin_markarian_shadrin_2013_hypercommutative_operad_homotopy}.

\begin{theorem}[{\cite[Theorem 1.3]{khoroshkin_markarian_shadrin_2013_hypercommutative_operad_homotopy}}]\label{thm:formality_BV/D}
  The operad $\BV/\Delta$ is formal and its cohomology is the hypercommutative operad $H^*(\BV/\Delta)=\H$. Moreover there is a direct quasi-isomorphism $\H\to\BV/\Delta$.
\end{theorem}

We now derive some consequences of this \namecref{thm:formality_BV/D} and \Cref{thm:Poisson_BV/D}

\begin{corollary}[{\cite[Example 1.3]{khoroshkin_markarian_shadrin_2013_hypercommutative_operad_homotopy}}]\label{cor:Poisson_complex_hypercommutative}
  The de Rham complex $\CdR{M}$ of a Poisson manifold $(M,\pi)$ has a hypercommutative algebra structure 
  obtained from the trivialized $BV$-algebra structure (\Cref{thm:Poisson_BV/D}) by restriction of scalars along the quasi-isomorphism in \Cref{thm:formality_BV/D}. It extends the commutative algebra structure given by the exterior product.
\end{corollary}

The following result can be found in \cite{barannikov_kontsevich_1998_frobenius_manifolds_formality,manin_1999_frobenius_manifolds_quantum,losev_shadrin_2007_zwiebach_invariants_getzler,park_2007_semiclassical_quantum_fields}.

\begin{corollary}\label{cor:Poisson_cohomology_hypercommutative}
  The de Rham cohomology $\HdR{M}$ of a Poisson manifold $(M,\pi)$ has a hypercommutative algebra structure induced by \Cref{cor:Poisson_complex_hypercommutative}. It extends the commutative algebra structure induced by the exterior product.
\end{corollary}

We can actually go further, deriving a slight strengthening of \cite[Theorem 3.7]{dotsenko_shadrin_vallette_2015_rham_cohomology_homotopy}.

\begin{corollary}\label{cor:Poisson_cohomology_H_infty}
  The de Rham cohomology $\HdR{M}$ of a Poisson manifold $(M,\pi)$ has a minimal $\H_\infty$-algebra structure extending \Cref{cor:Poisson_cohomology_hypercommutative} whose rectification is quasi-isomorphic to the hypercommutative algebra $\CdR{M}$ (\Cref{cor:Poisson_complex_hypercommutative}).
\end{corollary}

The operadic counterpart of \Cref{lem:EBV_is_BV/D} is the following \namecref{lem:BV/D_to_EBV}.

\begin{proposition}\label{lem:BV/D_to_EBV}
  There is a morphism of DG operads $\BV/\Delta\to\EBV$ extending that of \Cref{lem:BV_to_EBV} and satisfying $\phi_1\mapsto i$ and $\phi_n\mapsto 0$ for $n\geq 2$.
\end{proposition}

We can now start proving the results on Poisson manifolds announced in the introduction. We will recall the statements.

\Poissonfactorization*

\begin{proof}
  This is equivalent to \Cref{cor:Poisson_factors}. The first arrow of the factorization is \Cref{lem:BV/D_to_EBV} and the second one is given by \Cref{prop:Poisson_EBV}.
\end{proof}

\Poisson*

\begin{proof}
  Let $C$ be the cochain complex
  \begin{center}
    \begin{tikzpicture}
      \node (C0) at (0,0) {$\cdots$};
      \node (C1) at (1,0) {$0$};
      \node (C2) at (2,0) {$k$};
      \node (C3) at (3,0) {$k$};
      \node (C4) at (4,0) {$0$};
      \node (C5) at (5,0) {$\cdots$};
      \draw[->] (C0) -- (C1);
      \draw[->] (C1) -- (C2);
      \draw[->] (C2) -- (C3);
      \draw[->] (C3) -- (C4);
      \draw[->] (C4) -- (C5);
      \node at (-.2,.5) {\scriptsize degree};
      \node at (.9,.5) {$\scriptstyle -3$};
      \node at (1.9,.5) {$\scriptstyle -2$};
      \node at (2.9,.5) {$\scriptstyle -1$};
      \node at (4,.5) {$\scriptstyle 0$};
      \node at (5,.5) {$\cdots$};
      \node at (-.5,-.5) {\scriptsize generators};
      \node (i) at (2,-.5) {$i$};
      \node (di) at (3,-.505) {$d(i)$};
      \draw[|->] (i) -- (di);
    \end{tikzpicture}
  \end{center}
  We equip this complex with the SDR 
  \begin{center}
    \SDR{0}{C}[][][h]
  \end{center}
  determined by the formula
  \[hd(i)=i.\]
  We regard $C$ as an $\sym$-module concentrated in arity $1$ and the SDR as an $\sym$-module SDR. The operad $\EBV$ is the quotient of $\P=\C\amalg\F(C)$ by the DG ideal $I$ generated by 
  \begin{enumerate}
    \item\label{it:rel_com} $id(i)-d(i)i$.
    \item\label{it:diff_2} $(\mu^2\circ_1i)\cdot[()+(1\ 2)+(1\ 2\ 3)]+i\mu^2-\mu\circ_1(i\mu)\cdot[()+(2\ 3)+(1\ 3\ 2)]$
  \end{enumerate}
  These generators are weight-homogeneous of weights $2$ and $1$, respectively. The homotopy $h_C$ of \Cref{thm:operad_SDR} maps both generators to $0\in I$, by \Cref{thm:operad_SDR} \eqref{it:h_composition}. Indeed, $h\eqref{it:diff_2}=0$ is easy because \eqref{it:diff_2} has weight $1$ and $h(i)=0$ for degree reasons. For \eqref{it:rel_com}, 
  \begin{align*}
    h(id(i)-d(i)i) & =\frac{1}{2}h(i)d(i)+\frac{1}{2}ihd(i)-\frac{1}{2}hd(i)i+\frac{1}{2}d(i)h(i) \\
                   & =0+\frac{1}{2}i^2-\frac{1}{2}i^2+0                                           \\
                   & =0.
  \end{align*}
  
  Therefore, the composite
  \[\C\hookrightarrow\P\twoheadrightarrow\EBV\]
  is a quasi-isomorphism by \Cref{thm:ideal}. This composite is the inclusion in the statement.
\end{proof}

We now proceed with the corollaries. The first one is connected to \Cref{cor:Poisson_cohomology_hypercommutative}.

\trivialhypercommutativecohomology*

\begin{proof}
  This hypercommutative algebra structure on $\HdR{M}$ is the one in \Cref{cor:Poisson_cohomology_hypercommutative}.
  This \namecref{cor:trivial_hypercommutative_cohomology} follows from \Cref{cor:Poisson_factors} and from the fact that the composite $\C\to\H\to\BV/\Delta\to\EBV$ of the inclusion of the degree $0$ part $\C\subset\H$ with the maps in \Cref{thm:formality_BV/D} and \Cref{lem:BV/D_to_EBV} is the quasi-isomorphism in \Cref{thm:Poisson}.
\end{proof}

The second one is related to \Cref{cor:Poisson_complex_hypercommutative} and requires a more technical argument.

\trivialhypercommutativecohomologystrict*

\begin{proof}
  The first hypercommutative algebra structure on $\CdR{M}$ is the one in \Cref{cor:Poisson_complex_hypercommutative}.
  Both hypercommutative algebra structures can be obtained from two exact $BV$-algebra structures on $\CdR{M}$ with the same commutative product (the exterior product) by restriction of scalars along the map $\H\to\EBV$ which is the composition of the quasi-isomorphism $\H\to\BV/\Delta$ in \eqref{eq:Poisson_diagram}, see also \Cref{thm:formality_BV/D}, and the map $\BV/\Delta\to\EBV$ in \Cref{prop:Poisson_factorization,lem:BV/D_to_EBV}. In the first exact $BV$-algebra structure $i=i_\pi$ (\Cref{prop:Poisson_EBV}) and in the second one $i=0$. It therefore suffices to prove that these exact $BV$-algebra structures on $\CdR{M}$ are quasi-isomorphic.
  
  By \Cref{thm:Poisson} and \cite[Theorem 4.7.4]{hinich_1997_homological_algebra_homotopy}, the restriction of scalars along the inclusion $\C\subset\EBV$ induces an equivalence between the homotopy categories of exact $BV$-algebras and DG commutative algebras. Since the two exact $BV$-algebra structures in the previous paragraph have the same underlying DG commutative algebra structure, the corollary follows.
\end{proof}

The following corollary is a immediate consequence of the previous one.

\trivialhypercommutativecohomologyinfty*

The remaining corollary ensures the formality of the $\H$-algebra $\CdR{M}$ (\cite{drummond-cole_vallette_2013_minimal_model_batalin,khoroshkin_markarian_shadrin_2013_hypercommutative_operad_homotopy,dotsenko_shadrin_vallette_2015_rham_cohomology_homotopy}) when $M$ is a formal manifold.

\hypercommutativeformality*

\begin{proof}
  If $M$ is $\H$-formal then it is formal since $\C\subset\H$. Suppose now $M$ is formal. A zig-zag of commutative algebra quasi-isomorphisms between $\CdR{M}$ and $\HdR{M}$ is also a zig-zag of $\H$-algebra quasi-isomorphisms between the corresponding trivial $\H$-algebra structures. Now the result follows from \Cref{cor:trivial_hypercommutative_cohomology_strict}.
\end{proof}

The previous triviality and formality results extend to the operads which sit over $\EBV$. Let us start by defining the precise triviality notions.

\begin{definition}\label{def:trivial_algebras}
  Let $\P$ be a graded operad concentrated in degrees $\leq 0$ and $\P^0\subset\P$ its degree $0$ suboperad. A \emph{$\P$-algebra} is \emph{trivial} if it is quasi-isomorphic to the restriction of scalars of its underlying $\P^0$-algebra along the unique retraction $\P\to\P^0$.
\end{definition}

\begin{remark}
  If $\P$ is Koszul then so is $\P^0$. In this case, a $\P$-algebra $A$ is trivial if and only if the transferred $\P_\infty$-algebra structure on $H^*(A)$ is $\infty$-isomorphic to the restriction of scalars of the transferred $\P^0_\infty$-algebra structure along the morphism $\P_\infty\to\P^0_\infty$ induced by the retraction $\P\to\P^0$.
\end{remark}

\begin{remark}
  Let us specify \Cref{def:trivial_algebras} for some particular cases of interest.
  
  A (shifted) Lie algebra is trivial if it is abelian. A Gerstenhaber algebra is trivial if the underlying shifted Lie algebra structure is abelian. A $BV$-algebra is trivial if $\Delta=0$. 
  
  For the previous kinds of algebras $A$, we now explain how triviality can be detected in the transferred $\infty$-algebra $H^*(A)$.  
  A (shifted) Lie algebra $A$ is trivial if all operations vanish in the (shifted) $\Lie_\infty$-algebra $H^*(A)$. If $A$ is a Gerstenhaber algebra, the $\G_\infty$-algebra structure on $H^*(A)$ is defined by operations $m_{p_1,\dots,p_t}$, $t\geq 1$, $1\leq p_1\leq\cdots\leq p_t$, of arity $p_1+\cdots+p_t$ and degree $3-t-p_1-\cdots-p_t$ satisfying certain equations, see \cite{tamarkin_tsygan_2000_noncommutative_differential_calculus} and \cite[Proposition 16]{galvez-carrillo_tonks_vallette_2012_homotopy_batalin_vilkovisky}. Its underlying $\C_\infty$-algebra structure is given by the operations $m_{p_1}$ and triviality is equivalent to $m_{p_1,\dots,p_t}=0$ for $t\geq 2$. Similarly, if $A$ is a $BV$-algebra the $\BV_\infty$-algebra structure on $H^*(A)$ is given by operations $m^d_{p_1,\dots,p_t}$, $d\geq 0$, $t\geq 1$, $1\leq p_1\leq\cdots\leq p_t$, of arity $p_1+\cdots+p_t$ and degree $3-2d-t-p_1-\cdots-p_t$ satisfying certain equations, see \cite[Theorem 20]{galvez-carrillo_tonks_vallette_2012_homotopy_batalin_vilkovisky}. Its underlying $\C_\infty$-algebra structure is given by the operations $m^0_{p_1}$ and triviality corresponds to the vanishing of the remaining operations.
\end{remark}

\begin{corollary}\label{cor:P-trivial}
  Given a Poisson manifold $(M,\pi)$, a graded operad $\P$ concentrated in degrees $\leq 0$, and an operad morphism $\P\to\EBV$, the $\P$-algebra structure on $\CdR{M}$ obtained from \Cref{prop:Poisson_EBV} by restriction of scalars is quasi-isomorphic to the trivial $\P$-algebra structure.
\end{corollary}

\begin{proof}
  Consider the following diagram where the top vertical arrows are the inclusions of the degree $0$ suboperads and the bottom vertical arrows are the only possible retractions (hence the two squares commute),
  \[\begin{tikzcd}
      \P^0\ar[r]\ar[d]&\C\ar[d]\\
      \P\ar[r]\ar[d]&\EBV\ar[d]\ar[rdd, bend left, "\text{\Cref{prop:Poisson_EBV}}"]\\
      \P^0\ar[r]&\C\ar[rd,"\parbox{1cm}{\scriptsize exterior product}"']\\
      &&\E{\CdR{M}}
    \end{tikzcd}\]
  The triangle does not commute, but it induces quasi-isomorphic exact $BV$-algebra structures on $\CdR{M}$, see the proof of \Cref{cor:trivial_hypercommutative_cohomology_strict}, carried out in \Cref{sec:EBV}. Hence, they restrict to quasi-isomorphic $\P$-algebra structures. This is precisely what we had to prove.
\end{proof}

\begin{corollary}
  Let $(M,\pi)$ be a Poisson manifold:
  \begin{enumerate}
    \item The Lie algebra structure on $\CdR{M}$ given by the Koszul bracket is quasi-isomorphic to the abelian Lie algebra structure.
    \item The Gerstenhaber algebra structure on $\CdR{M}$ given by the Koszul bracket and the exterior product is quasi-isomorphic to the trivial Gerstenhaber algebra structure.
    \item The $BV$-algebra structure on $\CdR{M}$ given by the Lie derivative $\Delta=L_\pi$ and the exterior product is quasi-isomorphic to the trivial $BV$-algebra structure.
  \end{enumerate}
\end{corollary}

We can derive from here the Lie formality result in \cite{sharygin_talalaev_2008_lieformality_poisson_manifolds}.

\begin{corollary}\label{cor:talalaev}
  Let $(M,\pi)$ be a Poisson manifold. The Lie algebra structure on $\CdR{M}$ given by the Koszul bracket is formal.
\end{corollary}

\begin{proof}
  Since it is quasi-isomorphic to the abelian Lie algebra structure, it suffices to check that the latter is formal, but this is obvious because any complex over a field is formal.
\end{proof}

Another immediate consequence is the following formality \namecref{cor:general_formality}.

\begin{corollary}\label{cor:general_formality}
  Let $(M,\pi)$ be a Poisson manifold, $\P$ a graded operad concentrated in degrees $\leq 0$, $\P^0\subset\P$ its degree $0$ suboperad, and $\P\to\EBV$ an operad morphism. We use it together with \Cref{prop:Poisson_EBV} to equip $\CdR{M}$ with a $\P$-algebra structure. It has an underlying $\P^0$-algebra structure obtained by restriction of scalars along the inclusion of the degree $0$ part $\P^0\subset\P$. The following statements hold:
  \begin{enumerate}
    \item If $M$ is formal then it is also $\P^0$-formal.
    \item $M$ is $\P$-formal if and only if it is $\P^0$-formal.
  \end{enumerate}
\end{corollary}

\begin{corollary}\label{cor:G_BV_formality}
  Let $(M,\pi)$ be a Poisson manifold. The following statements are equivalent: 
  \begin{enumerate}
    \item $M$ is a formal manifold.
          \item\label{it:G_formality} $M$ is $\G$-formal with respect to the Gerstenhaber algebra structure on $\CdR{M}$ given by the exterior product and the Koszul bracket.
    \item $M$ is $\BV$-formal with respect to the $BV$-algebra structure on $\CdR{M}$ given by the exterior product and the Lie derivative $\Delta=L_\pi$.
  \end{enumerate}
\end{corollary}

Notice that \Cref{cor:G_BV_formality} \eqref{it:G_formality} does not follow from \Cref{cor:talalaev} since the former quasi-isomorphism is compatible with the exterior product, but the latter does not take it into account.

\section{Jacobi manifolds}\label{sec:Jacobi}

The following generalization of $BV$-algebras was introduced in \cite{kravchenko_2000_deformations_batalin_vilkovisky}.

\begin{definition}\label{def:commutative_BV_infty_algebra}
  A \emph{commutative $BV_\infty$-algebra} is a DG commutative algebra $A$ equipped with differential operators $\Delta_n\colon A\to A$, $n\geq 1$, of degree $1-2n$ and order $\leq n+1$ such that:
  \[[d,\Delta_n]+\sum_{i=1}^{n-1}\Delta_i\Delta_{n-i}=0,\qquad n\geq 1.\]
\end{definition}

\begin{remark}\label{rem:BV_and_BV_infty}
  A $BV$-algebra is the same as a commutative $BV_\infty$-algebra with $\Delta_n=0$ for $n\geq 2$.  
\end{remark}

Recall that a \emph{Jacobi manifold} $(M,\pi,\eta)$ is a manifold $M$ equipped with a bivector field $\pi\in\Gamma(\Lambda^2T(M))$ and a vector field $\eta\in \Gamma(T(M))$ such that the following relations hold in $\Gamma(\Lambda^*T(M))$:
\[[\pi,\pi]=2\eta\pi,\qquad [\pi,\eta]=0.\]
In particular, a Poisson manifold is a Jacobi manifold with $\eta=0$.

\begin{theorem}[{\cite[Theorem 4.3]{dotsenko_shadrin_vallette_2015_rham_cohomology_homotopy}}]\label{rem:Jacobi_BV_infty}
  The de Rham complex $\CdR{M}$ of a Jacobi manifold is a commutative $BV_\infty$-algebra with the exterior differential, the exterior product,
  \begin{align*}
    \Delta_1 & =L_\pi,                       & 
    \Delta_2 & =-i_\eta i_\pi=-i_{\eta\pi},
  \end{align*}
  and $\Delta_n=0$ for $n\geq 3$.
\end{theorem}

\begin{remark}\label{rem:fix_sign}
  The sign of $\Delta_2$ here is opposite to that in \cite[Theorem 4.3]{dotsenko_shadrin_vallette_2015_rham_cohomology_homotopy}. This is due to a sign discrepancy between \cite[Proposition 3.6 and formula (4)]{dotsenko_shadrin_vallette_2015_rham_cohomology_homotopy} and \cite[p.~266]{koszul_1985_crochet_schoutennijenhuis_cohomologie}, see also \cite[\S3]{cartier_1994_fundamental_techniques_theory}.
\end{remark}

We now consider the algebraic structure on the de Rham complex of a Jacobi manifold induced by the interior product with the structure polyvector fields.

\begin{definition}\label{def:Jacobi}
  A \emph{Jacobi algebra} is a DG commutative algebra $A$ equipped with two differential operators: 
  \begin{itemize}
    \item $i\colon A\to A$ of degree $-2$ and order $\leq 2$,
    \item $j\colon A\to A$ of degree $-1$ and order $\leq 1$,
  \end{itemize}
  such that:
  \begin{enumerate}
    \item\label{it:Jacobi_[i,[d,i]]} $[i,[i,d]]=-2ji$.
    \item\label{it:Jacobi_[i,j]} $[i,j]=0$.
    \item\label{it:Jacobi_[[d,i],j]} $[j,[i,d]]=0$.
    \item\label{it:Jacobi_j2} $j^2=0$.
  \end{enumerate}
  The first three equations are equivalent to:
  \begin{enumerate}
    \item $idi-ji=\frac{i^2d+di^2}{2}$.
    \item $ij=ji$.
    \item $jid+idj=jdi+dij$.
  \end{enumerate}
\end{definition}

\begin{remark}\label{rem:EBV_is_Jacobi}
  An exact $BV$-algebra is the same as a Jacobi algebra with $j=0$.
\end{remark}

\begin{proposition}\label{prop:Jacobi_J}
  The de Rham complex $\CdR{M}$ of a Jacobi manifold $(M,\pi,\eta)$ is a Jacobi algebra with the exterior differential, the exterior product, $i=i_\pi$, and $j=i_\eta$.
\end{proposition}

\begin{proof}
  Since $\eta$ is a vector field, $i_\eta$ is a differential operator of degree $-1$ and order $\leq 1$, and $i_\pi$ is a differential operator of degree $-2$ and order $\leq 2$ just like in the proof of \Cref{prop:Poisson_EBV}. Moreover, if we likewise use the equations in \cite[\S3]{cartier_1994_fundamental_techniques_theory}, we get  
  \begin{align*}
    [i_\pi,[i_\pi,d]]  & =[i_\pi,L_\pi]=-[L_\pi,i_\pi]=-i_{[\pi,\pi]}=-2i_{\eta\pi}=-2i_\eta i_\pi, \\
    [i_\eta,[i_\pi,d]] & =[i_\eta,L_\pi]=-[L_\pi,i_\eta]=-i_{[\pi,\eta]}=0,
  \end{align*}
  by the equations $[\pi,\pi]=2\eta\pi$ and $[\pi,\eta]=0$, and
  \begin{align*}
    [i_\pi,i_\eta] & =i_{\pi}i_{\eta}-i_{\eta}i_{\pi}=i_{\pi\eta}-i_{\eta\pi}=0, \\
    i_\eta^2       & =i_{\eta^2}=0,
  \end{align*}
  since $\CdR{M}$ is graded commutative.
\end{proof}

\begin{proposition}\label{prop:Jacobi_CBV}
  Any Jacobi algebra is a commutative $BV_\infty$-algebra with $\Delta_1=[i,d]$, $\Delta_2=-ji$, and $\Delta_n=0$ for $n\geq 3$.
\end{proposition}

\begin{proof}
  The operator $\Delta_1$ is a degree $-1$ and order $\leq 2$ differential operator like in the proof of \Cref{lem:EBV_is_BV}, and $ji$ is a differential operator of degree $-3$ and order $\leq 3$ because of the degrees and orders of $i$ and $j$ add up to these numbers. Since $\Delta_n=0$ for $n\geq 3$, we only have to check the equation in \Cref{def:commutative_BV_infty_algebra} for $n=1,2,3,4$:
  \begin{enumerate}
    \item $[d,\Delta_1]=[d,[i,d]]=0$ since $[d,-]$ is the differential of the endomorphism DG algebra.
    \item $[d,\Delta_2]+\Delta_1^2=-[d,ji]+[i,d]^2=-dji-jid+idid-di^2d+didi=0$ by \Cref{def:Jacobi} \eqref{it:Jacobi_[i,[d,i]]}.
    \item $\Delta_1\Delta_2+\Delta_2\Delta_1 =[\Delta_1,\Delta_2]=-[[i,d],ji]=-[[i,d],j]i+j[[i,d],i]=2j^2i=0$ by \Cref{def:Jacobi} \eqref{it:Jacobi_[i,[d,i]]}, \eqref{it:Jacobi_[[d,i],j]}, and \eqref{it:Jacobi_j2}.
    \item $\Delta_2^2=jiji=ij^2i=0$ by \Cref{def:Jacobi} \eqref{it:Jacobi_[i,j]} and \eqref{it:Jacobi_j2}.
  \end{enumerate}
\end{proof}

We have a commutative $BV_\infty$-version of \Cref{def:BV/D_algebra}, compare \cite[Theorem 2.1]{dotsenko_shadrin_vallette_2015_rham_cohomology_homotopy}.

\begin{definition}\label{def:CBV/D_algebra}
  A \emph{trivialized commutative $BV_\infty$-algebra} or \emph{commutative $BV_\infty$-algebra with Hodge-to-de-Rham degeneration data} is a commutative  $BV_\infty$-algebra $A$ equipped with graded vector space morphisms $\phi_n\colon A\to A$, $n\geq 1$, of degree $-2n$, such that the following formula holds in $\End{A}[[z]]$,
  \[e^{\phi(z)}de^{-\phi(z)}=d+\sum_{i=n}^\infty\Delta_nz^n,\qquad \phi(z)=\sum_{n=1}^\infty\phi_nz^n.\]
  Equivalently, for each $n\geq 1$,
  \begin{equation}\label{eq:phi_CBV/D}
    \sum_{p=1}^n\frac{1}{p!}\sum_{\substack{q_1+\cdots+q_p=n\\q_1,\dots,q_p\geq 1}}
    [\phi_{q_1},\dots[\phi_{q_{p-1}},[\phi_{q_p},d]]\dots]=\Delta_n.
  \end{equation}
\end{definition}

\begin{remark}\label{rem:BV/D_is_CBV/D}
  A trivialized $BV$-algebra is the same as a trivialized commutative  $BV_\infty$-algebra with $\Delta_n=0$ for $n\geq 2$.
\end{remark}

\begin{theorem}[{\cite[Theorems 2.1 and 4.5]{dotsenko_shadrin_vallette_2015_rham_cohomology_homotopy}}]\label{thm:Jacobi_CBV/D}
  The de Rham complex $\CdR{M}$ of a Jacobi manifold $(M,\pi,\eta)$ is a trivialized commutative $BV_\infty$-algebra with the exterior differential, the exterior product, and
  \begin{align*}
    \Delta_1 & =L_\pi,                      & 
    \Delta_2 & =-i_\eta i_\pi=-i_{\eta\pi}, & 
    \phi_1   & =i_\pi.
  \end{align*}
  The remaining operations are zero.
\end{theorem}

See \Cref{rem:fix_sign} for signs.

\begin{proposition}\label{prop:Jacobi_CBV/D}
  Any Jacobi algebra is a trivialized commutative $BV_\infty$-algebra with $\Delta_1=[i,d]$, $\Delta_2=-ji$, $\Delta_n=0$ for $n\geq 3$, $\phi_1=i$, and $\phi_n=0$ for $n\geq 2$.
\end{proposition}

\begin{proof}
  After \Cref{prop:Jacobi_CBV}, we only need to check \eqref{eq:phi_CBV/D}. Since the only non-vanishing $\Delta_n$ and $\phi_n$ are $\Delta_1$, $\Delta_2$, and $\phi_1$, this reduces to $[i,d]=\Delta_1$, which holds by definition, $\frac{1}{2}[i,[i,d]]=-ji$, which is \Cref{def:Jacobi} \eqref{it:Jacobi_[i,[d,i]]}, and $[i,\dots,[i,[i,d]]]=0$, which is a consequence of 
  \begin{align*}
    [i,[i,[i,d]]] & =-2[i,ji]         \\
                  & =-2[i,j]i-2j[i,i] \\
                  & =0.
  \end{align*}
  Here we use again \Cref{def:Jacobi} \eqref{it:Jacobi_[i,[d,i]]} and \eqref{it:Jacobi_[i,j]}.
\end{proof}

\begin{corollary}\label{cor:Jacobi_factors}
  The trivialized commutative $BV_\infty$-algebra structure on the de Rham complex $\CdR{M}$ of a Jacobi manifold $(M,\pi,\eta)$ (\Cref{thm:Jacobi_CBV/D}) is determined by its Jacobi algebra structure (\Cref{prop:Jacobi_J}) and \Cref{prop:Jacobi_CBV/D}.
\end{corollary}

We now consider the operads corresponding to the new classes of algebras.

\begin{definition}\label{def:CBV_operad}
  The \emph{commutative $BV_\infty$ operad} $\CBV$ is the DG operad generated by $\mu\in\CBV(2)^0$ and $\Delta_n\in\CBV(1)^{1-2n}$, $n\geq 1$, with 
  \[\mu\cdot (1\ 2)=\mu\]
  subject to the relations:
  \begin{enumerate}
    \item $d(\mu)=0$.
    \item $\mu\circ_1\mu-\mu\circ_2\mu=0$.
          \item\label{it:dD_n} $d(\Delta_n)+\sum_{i=1}^{n-1}\Delta_i\Delta_{n-i}=0$, $n\geq 1$,
          \item\label{it:D_n_diff_op} $\sum_{\substack{p+q=n+1\\p\geq1}}(-1)^p\sum_{\sigma\in\shuffle{p}{q}}(\mu^{q+1}\circ_1(\Delta_n\mu^{p}))\cdot\sigma^{-1}=0$, $n\geq 1$.
  \end{enumerate}
  Here \eqref{it:dD_n} corresponds to the equation in \Cref{def:commutative_BV_infty_algebra} and \eqref{it:D_n_diff_op} is the equation in \Cref{def:differential_operator}. 
  This operad is not a minimal resolution of another one, it is not even cofibrant, but we still denote it by $\CBV$. Like in \Cref{def:BV_operad}, the commutative operad $\C\subset\CBV$ is the suboperad generated by $\mu$.
\end{definition}

\begin{proposition}\label{prop:CBV_to_BV}
  There is a morphism of DG operads $\CBV\to\BV$ given by $\mu\mapsto\mu$, $\Delta_1\mapsto\Delta$, and $\Delta_n\mapsto 0$ for $n\geq 2$. Moreover, this morphism is a quasi-isomorphism.
\end{proposition}

This is the operadic counterpart of \Cref{rem:BV_and_BV_infty}. It is a quasi-isomorphism by \cite[Theorem 5.3.1]{campos_merkulov_willwacher_2016_frobenius_properad_koszul}.

\begin{definition}
  The \emph{Jacobi operad} $\J$ is the DG operad generated by $\mu\in\J(2)^0$, $i\in\J(1)^{-2}$ and $j\in\J(1)^{-1}$ with 
  \[\mu\cdot (1\ 2)=\mu\]
  subject to the relations:
  \begin{enumerate}
    \item $d(\mu)=0$.
    \item $\mu\circ_1\mu=\mu\circ_2\mu$.
          \item\label{it:Jacobi_commutator} $id(i)=d(i)i+2ji$.
    \item $ij=ji$.
    \item $d(i)j+jd(i)=0$.
          \item\label{it:Jacobi_j2_operad} $j^2=0$.
          \item\label{it:j_order_1} $(\mu\circ_1j)\cdot[()+(1\ 2)]=j\mu$.
          \item\label{it:i_order_2} $(\mu^2\circ_1i)\cdot[()+(1\ 2)+(1\ 2\ 3)]+i\mu^2=\mu\circ_1(i\mu)\cdot[()+(2\ 3)+(1\ 3\ 2)]$.
  \end{enumerate}
  Here \eqref{it:Jacobi_commutator}--\eqref{it:Jacobi_j2_operad} correspond to relations \eqref{it:Jacobi_[i,[d,i]]}--\eqref{it:Jacobi_j2} in \Cref{def:Jacobi}, and \eqref{it:j_order_1} and \eqref{it:i_order_2} are \Cref{def:differential_operator} for $n=1,2$, respectively. 
  Like in \Cref{def:EBV_operad}, the DG suboperad generated by $\mu$ is the commutative operad $\C\subset\J$.
\end{definition}

\begin{proposition}\label{prop:Jacobi_to_EBV}
  There is a morphism of DG operads $\J\to\EBV$ given by $\mu\mapsto\mu$, $i\mapsto i$, and $j\mapsto 0$.
\end{proposition}

This is the operad version of \Cref{rem:EBV_is_Jacobi}. The following \namecref{lem:CBV_to_J} is the operadic counterpart of \Cref{prop:Jacobi_CBV}.

\begin{proposition}\label{lem:CBV_to_J}
  There is a morphism of DG operads $\CBV\to\J$ given by $\mu\mapsto\mu$, $\Delta_1\mapsto -d(i)$, $\Delta_2\mapsto -ji$, and $\Delta_n\mapsto 0$ for $n\geq 3$.
\end{proposition}

The operad for trivialized commutative $BV_\infty$-algebras (\Cref{def:CBV/D_algebra}) is the following one.

\begin{definition}\label{def:trivialized_CBV_operad}
  The \emph{trivialized commutative $BV_\infty$ operad} $\CBV/\Delta$ is the DG operad generated by $\mu\in(\CBV/\Delta)(2)^0$, $\Delta_n\in(\CBV/\Delta)(1)^{1-2n}$, and $\phi_n\in(\CBV/\Delta)(1)^{-2n}$, $n\geq 1$, with 
  \[\mu\cdot (1\ 2)=\mu\]
  subject to the relations in \Cref{def:CBV_operad} and 
  \[d(\phi_n)+\sum_{p=2}^n\frac{1}{p!}\sum_{\substack{q_1+\cdots+q_p=n\\q_1,\dots,q_p\geq 1}}
    [\phi_{q_1},\dots[\phi_{q_{p-1}},d(\phi_{q_p})]\dots]=-\Delta_n,\qquad n\geq 1.\]
  This relation corresponds to \eqref{eq:phi_CBV/D}. 
  Clearly, $\CBV\subset\CBV/\Delta$ is a suboperad. This inclusion $\CBV\rightarrowtail \CBV/\Delta$ is a cofibration in $\Operads$. Indeed, it is a standard cofibration in the sense of \cite[\S6.4]{hinich_1997_homological_algebra_homotopy}.
\end{definition}

\begin{proposition}\label{prop:CBV/D_to_BV/D}
  There is a DG operad morphism $\CBV/\Delta\to\BV/\Delta$ given by $\mu\mapsto\mu$, $\Delta_1\mapsto\Delta$, $\Delta_n\mapsto 0$ for $n\geq 2$, and $\phi_n\mapsto\phi_n$ for $n\geq 1$. Moreover, this morphism is a quasi-isomorphism.
\end{proposition}

\begin{proof}
  The morphism is the operadic version of \Cref{rem:BV/D_is_CBV/D}. We have a commutative square 
  \begin{center}
    \begin{tikzcd}
      \CBV\arrow[r,"\sim"]\arrow[tail,d] & \BV\arrow[tail,d] \\
      \CBV/\Delta\arrow[r]                  & \BV/\Delta
    \end{tikzcd}
  \end{center}
  where the bottom horizontal arrow is the map in the statement, the top horizontal arrow is the quasi-isomorphism in \Cref{prop:CBV_to_BV}, and the vertical arrows are the inclusions in \Cref{def:BV/D_operad,def:trivialized_CBV_operad}, which are cofibrations in $\Operads$. The square is clearly a push-out in $\Operads$, it suffices to look at the presentations in these definitions. The model category $\Operads$ is left proper by \cite[Theorem 3.1.10]{hackney_robertson_yau_2016_relative_left_properness}, hence the bottom horizontal map is also a quasi-isomorphism.
\end{proof}

The following formality \namecref{cor:formality_CBV/D} uses also \Cref{thm:formality_BV/D}.

\begin{corollary}\label{cor:formality_CBV/D}
  The DG operad $\CBV/\Delta$ is formal and its cohomology is the hypercommutative operad $H^*(\CBV/\Delta)=\H$.
\end{corollary}

In this case, we do not know if we can define a direct quasi-isomorphism $\H\to\CBV/\Delta$. However, we have an explicit zig-zag
\[\CBV/\Delta\stackrel{\sim}{\longrightarrow}\BV/\Delta\stackrel{\sim}{\longleftarrow}\H\]
given by \Cref{prop:CBV/D_to_BV/D} and \Cref{thm:formality_BV/D}.

\begin{corollary}\label{cor:H_infty_to_CBV/D}
  Let $\H_\infty$ be the minimal model of the Koszul operad $\H$. There is a quasi-isomorphism of DG operads $\H_\infty\to\CBV/\Delta$.
\end{corollary}

\begin{proof}
  Since $\H_\infty$ is a cofibrant replacement of $\H$ in $\Operads$, this \namecref{cor:H_infty_to_CBV/D} follows from \Cref{cor:formality_CBV/D}.
\end{proof}

We can use the previous results to enhance the de Rham complex and cohomology of a Jacobi manifold with the structure of a (homotopy) hypercommutative algebra, extending their well-known (homotopy) commutative algebra structure. However, we will soon see that all those structures are trivial, like for Poisson manifolds.

\begin{corollary}\label{cor:Jacobi_complex_hypercommutative}
  The de Rham complex $\CdR{M}$ of a Jacobi manifold $(M,\pi,\eta)$ has an $\H_\infty$-algebra structure obtained by restriction of scalars of \Cref{thm:Jacobi_CBV/D} along \Cref{cor:H_infty_to_CBV/D}, whose underlying $\C_\infty$-algebra is the usual plain commutative algebra structure.
\end{corollary}

\begin{corollary}\label{cor:Jacobi_cohomology_hypercommutative}
  The de Rham cohomology $\HdR{M}$ of a Jacobi manifold $(M,\pi,\eta)$ has a hypercommutative algebra structure, induced by \Cref{cor:Jacobi_complex_hypercommutative}, extending the usual graded commutative algebra structure.
\end{corollary}

Like in the Poisson case (\Cref{cor:Poisson_cohomology_H_infty}), we can deduce a slight improvement of \cite[Theorem 4.5]{dotsenko_shadrin_vallette_2015_rham_cohomology_homotopy}.

\begin{corollary}\label{cor:Jacobi_cohomology_H_infty}
  The de Rham cohomology $\HdR{M}$ of a Jacobi manifold $(M,\pi,\eta)$ has a minimal $\H_\infty$-algebra structure, obtained from \Cref{cor:Jacobi_complex_hypercommutative} by homotopy transfer, extending \Cref{cor:Jacobi_cohomology_hypercommutative}  whose rectification is strictly quasi-isomorphic to the $\H_\infty$-algebra $\CdR{M}$ in \Cref{cor:Jacobi_complex_hypercommutative}.
\end{corollary}

The counterpart of \Cref{prop:Jacobi_CBV/D} for operads is the following \namecref{lem:CBV/D_to_J}.

\begin{proposition}\label{lem:CBV/D_to_J}
  There is a morphism of DG operads $\CBV/\Delta\to\J$ extending \Cref{lem:CBV_to_J} and satisfying $\phi_1\mapsto i$ and $\phi_n\mapsto 0$ for $n\geq 2$.
\end{proposition}

We now present the operadic version of \Cref{cor:Jacobi_factors}.

\begin{proposition}\label{prop:Jacobi_factorization}
  Given a Jacobi manifold $(M,\pi,\eta)$, the morphism $\CBV/\Delta\to\E{\CdR{M}}$ determined by \Cref{thm:Jacobi_CBV/D} factors as \[\CBV/\Delta\longrightarrow\J\longrightarrow\E{\CdR{M}},\]
  where the first arrow is the one in \Cref{lem:CBV/D_to_J} and the second arrow is given by \Cref{prop:Jacobi_J}.
\end{proposition}

The following \namecref{thm:Jacobi} is the main result of this section.

\begin{theorem}\label{thm:Jacobi}
  The DG operad inclusion $\C\subset\J$ is a quasi-isomorphism.
\end{theorem}

\begin{proof}
  Consider the cochain complex $C=C_1\oplus C_2$, where 
  \begin{center}
    \begin{tikzpicture}
      \node at (-1,0) {$C_1\colon$};
      \node (C0) at (0,0) {$\cdots$};
      \node (C1) at (1,0) {$0$};
      \node (C2) at (2,0) {$k$};
      \node (C3) at (3,0) {$k$};
      \node (C4) at (4,0) {$0$};
      \node (C5) at (5,0) {$\cdots$};
      \draw[->] (C0) -- (C1);
      \draw[->] (C1) -- (C2);
      \draw[->] (C2) -- (C3);
      \draw[->] (C3) -- (C4);
      \draw[->] (C4) -- (C5);
      \node at (-.2,.5) {\scriptsize degree};
      \node at (.9,.5) {$\scriptstyle -3$};
      \node at (1.9,.5) {$\scriptstyle -2$};
      \node at (2.9,.5) {$\scriptstyle -1$};
      \node at (4,.5) {$\scriptstyle 0$};
      \node at (5,.5) {$\cdots$};
      \node at (-.5,-.5) {\scriptsize generators};
      \node (i) at (2,-.5) {$i$};
      \node (di) at (3,-.505) {$d(i)$};
      \draw[|->] (i) -- (di);
    \end{tikzpicture}
  \end{center}
  \begin{center}
    \begin{tikzpicture}
      \node at (-1,0) {$C_2\colon$};
      \node (C0) at (0,0) {$\cdots$};
      \node (C1) at (1,0) {$0$};
      \node (C2) at (2,0) {$k$};
      \node (C3) at (3,0) {$k$};
      \node (C4) at (4,0) {$0$};
      \node (C5) at (5,0) {$\cdots$};
      \draw[->] (C0) -- (C1);
      \draw[->] (C1) -- (C2);
      \draw[->] (C2) -- (C3);
      \draw[->] (C3) -- (C4);
      \draw[->] (C4) -- (C5);
      \node at (-.2,.5) {\scriptsize degree};
      \node at (.9,.5) {$\scriptstyle -2$};
      \node at (1.9,.5) {$\scriptstyle -1$};
      \node at (3,.5) {$\scriptstyle 0$};
      \node at (4,.5) {$\scriptstyle 1$};
      \node at (5,.5) {$\cdots$};
      \node at (-.5,-.5) {\scriptsize generators};
      \node (i) at (2,-.5) {$j$};
      \node (di) at (3,-.5) {$d(j)$};
      \draw[|->] (i) -- (di);
    \end{tikzpicture}
  \end{center}
  In this case we have an SDR
  \begin{center}
    \SDR{0}{C}[][][h]
  \end{center}
  determined by the formulas
  \[hd(i)=i,\qquad hd(j)=j.\]
  Necessarily $h(i)=0$ and $h(j)=0$.
  We regard $C$ as an $\sym$-module concentrated in arity $1$. Hence the previous SDR is an $\sym$-module SDR. The operad $\J$ is the quotient of $\P=\C\amalg\F(C)$ by the DG ideal $I$ generated by 
  \begin{enumerate}
    \item $id(i)-d(i)i-2ji$.
          \item\label{it:j_i_commutator} $ij-ji$.
          \item\label{it:j_d(i)_commutator} $d(i)j+jd(i)$.
    \item $j^2$.
    \item $\mu\circ_1j\cdot[()+(1\ 2)]-j\mu$.
    \item $(\mu^2\circ_1i)\cdot[()+(1\ 2)+(1\ 2\ 3)]+i\mu^2-\mu\circ_1(i\mu)\cdot[()+(2\ 3)+(1\ 3\ 2)]$.
  \end{enumerate}
  All these generators are weight-homogeneous. The first four generators have weight $2$, and the two remaining ones have weight $1$. The homotopy $h_C$ of \Cref{thm:operad_SDR} maps all generators to $0\in I$, except for $h_C\eqref{it:j_d(i)_commutator}=\frac{1}{2}\eqref{it:j_i_commutator}\in I$. Indeed, by \Cref{thm:operad_SDR} \eqref{it:h_composition},
  \begin{align*}
    h(id(i)-d(i)i-2ji) & =\frac{1}{2}h(i)d(i)+\frac{1}{2}ihd(i)-\frac{1}{2}hd(i)i+\frac{1}{2}d(i)h(i)-h(j)i+jh(i) \\
                       & =0+\frac{1}{2}i^2-\frac{1}{2}i^2+0+0+0                                                   \\
                       & =0,                                                                                      \\
    h(ij-ji)           & =\frac{1}{2}h(i)j+\frac{1}{2}ih(j)-\frac{1}{2}h(j)i+\frac{1}{2}jh(i)                     \\
                       & =0+0-0+0                                                                                 \\
                       & =0,                                                                                      \\
    h(d(i)j+jd(i))     & =\frac{1}{2}hd(i)j-\frac{1}{2}d(i)h(j)+\frac{1}{2}h(j)d(i)-\frac{1}{2}jhd(i)             \\
                       & =\frac{1}{2}ij-0+0-\frac{1}{2}ji                                                         \\
                       & =\frac{1}{2}(ij-ji),                                                                     \\
    h(j^2)             & =\frac{1}{2}h(j)j-\frac{1}{2}jh(j)                                                       \\
                       & =0-0                                                                                     \\
                       & =0.
  \end{align*}
  We omit the equations for the weight $1$ relations. They are very easy because $h(i)=h(j)=0$. 
  Therefore, by \Cref{thm:ideal}, the composite
  \[\C\hookrightarrow\P\twoheadrightarrow\J\]
  is a quasi-isomorphism, and this composite is the inclusion in the statement.
\end{proof}

As corollaries, we obtain the expected triviality results for the (homotopy) hypercommutative algebra structures on the de Rham complex and cohomology of Jacobi manifolds.

\begin{corollary}
  For any Jacobi manifold $(M,\pi,\eta)$, the hypercommutative algebra structure on $\HdR{M}$ defined in \Cref{cor:Jacobi_cohomology_hypercommutative} is trivial.
\end{corollary}

\begin{proof}
  This follows from \Cref{prop:Jacobi_factorization} and from the fact that the composite $\C\to\H_\infty\to\CBV/\Delta\to\J$ of the inclusion of the degree $0$ part $\C\subset\H$ with the maps in \Cref{cor:H_infty_to_CBV/D} and \Cref{lem:CBV/D_to_J} is the quasi-isomorphism in \Cref{thm:Jacobi}.
\end{proof}

\begin{corollary}
  For any Jacobi manifold $(M,\pi,\eta)$, the $\H_\infty$-algebra structure on $\CdR{M}$ defined in \Cref{cor:Jacobi_complex_hypercommutative} is quasi-isomorphic to the trivial one.
\end{corollary}

\begin{proof}
  We proceed like in the proof of \Cref{cor:trivial_hypercommutative_cohomology_strict}, carried out in \Cref{sec:EBV}. Both $\H_\infty$-algebra structures can be obtained from two Jacobi algebra structures on $\CdR{M}$ with the same commutative product (the exterior product) by restriction of scalars along the map $\H_\infty\to\J$ which is the composition of the quasi-isomorphism $\H_\infty\to\CBV/\Delta$ in \Cref{cor:H_infty_to_CBV/D} and the map $\CBV/\Delta\to\J$ in \Cref{lem:CBV/D_to_J}. In the first Jacobi algebra structure $i=i_\pi$ and $j=i_\eta$ (\Cref{prop:Jacobi_J}), and in the second one $i=0$ and $j=0$. It therefore suffices to prove that these Jacobi algebra structures on $\CdR{M}$ are quasi-isomorphic. The argument for this is identical to the last paragraph of the proof of \Cref{cor:trivial_hypercommutative_cohomology_strict}, carried out in \Cref{sec:EBV}, replacing $\EBV$ with $\J$.
\end{proof}

\begin{corollary}
  Given a Jacobi manifold $(M,\pi,\eta)$, the minimal $\H_\infty$-algebra structure on $\HdR{M}$ defined in \Cref{cor:Jacobi_cohomology_H_infty} is $\infty$-isomorphic to the trivial one.
\end{corollary}

The following two corollaries can be applied to $\P=\CBV$ the commutative $BV_\infty$-operad, see \Cref{lem:CBV/D_to_J}. In this case $\P^0=\C$ is the commutative operad, hence $\P^0$-formal means just formal. \Cref{def:trivial_algebras} applied to $\CBV$ says that a commutative $BV_\infty$-algebra is trivial if $\Delta_n=0$ for $n\geq 1$.

\begin{corollary}
  Given a Jacobi manifold $(M,\pi,\eta)$, a graded operad $\P$ concentrated in degrees $\leq 0$, and an operad morphism $\P\to\J$, the $\P$-algebra structure on $\CdR{M}$ obtained from \Cref{prop:Jacobi_J} by restriction of scalars is quasi-isomorphic to the trivial $\P$-algebra structure.
\end{corollary}

The proof is like that of \Cref{cor:P-trivial}, replacing $\EBV$ with $\J$ and invoking some of the previous results of this section.

\begin{corollary}
  Let $(M,\pi,\eta)$ be a Jacobi manifold, $\P$ a graded operad concentrated in degrees $\leq 0$, $\P^0\subset\P$ its degree $0$ suboperad, and $\P\to\J$ an operad morphism. We use it together with \Cref{prop:Jacobi_J} to equip $\CdR{M}$ with a $\P$-algebra structure. It has an underlying $\P^0$-algebra structure obtained by restriction of scalars along the inclusion of the degree $0$ part $\P^0\subset\P$. The following statements hold:
  \begin{enumerate}
    \item If $M$ is formal then it is also $\P^0$-formal.
    \item $M$ is $\P$-formal if and only if it is $\P^0$-formal.
  \end{enumerate}
\end{corollary}

\section{Generalized Poisson manifolds}\label{sec:generalized_poisson}

A \emph{generalized Poisson manifold} is a graded manifold or supermanifold $M$ with an even polyvector field $P$ satisfying $[P,P]=0$, see \cite[\S4]{khudaverdian_voronov_2008_higher_poisson_brackets} and \cite[Definition 4.1]{braun_lazarev_2013_homotopy_bv_algebras}.

\begin{proposition}[{\cite{khudaverdian_voronov_2008_higher_poisson_brackets},\cite[Definition 4.2]{braun_lazarev_2013_homotopy_bv_algebras}, \cite[Example 2.3]{bashkirov_voronov_2017_bv_formalism_l_}}]\label{graded_manifold_CBV}
  The de Rham complex $\CdR{M}$ of a graded manifold or supermanifold $M$ equipped with a generalized Poisson structure of the form $P=P_1+P_2+\cdots$ carries a canonical commutative $BV_\infty$-algebra structure defined by the Lie derivatives $\Delta_n=L_{P_{n}}$, $n\geq 1$.
\end{proposition}

Here we use the notation in \cite{braun_lazarev_2013_homotopy_bv_algebras}, not in \cite{bashkirov_voronov_2017_bv_formalism_l_}, where the indices of the components of $P$ are shifted by $+1$.

\begin{definition}\label{def:exact_commutative_BV_infty_algebra}
  An \emph{exact commutative $BV_\infty$-algebra} is a DG commutative algebra $A$ equipped with differential operators $i_n\colon A\to A$, $n\geq 1$, of degree $-2n$ and order $\leq n+1$ such that:
  \begin{enumerate}
    \item\label{it_ECB_infty_1} $[i_n,i_m]=0$, $n,m\geq 1$.
    \item\label{it_ECB_infty_2} $\sum_{j=1}^{n-1}[i_j,[i_{n-j},d]]=0$, $n\geq 2$.
  \end{enumerate}
\end{definition}

\begin{proposition}\label{prop:supermanifold_ECBV}
  The de Rham complex $\CdR{M}$ of a graded manifold or supermanifold $M$ equipped with a generalized Poisson structure of the form $P=P_1+P_2+\cdots$ carries a canonical exact commutative $BV_\infty$-algebra structure defined by the interior products $i_n=i_{P_{n}}$, $n\geq 1$.
\end{proposition}

This follows like \Cref{prop:Poisson_EBV,prop:Jacobi_J}, compare \cite[\S4]{braun_lazarev_2013_homotopy_bv_algebras} and \cite[Example 2.3]{bashkirov_voronov_2017_bv_formalism_l_}.

\begin{proposition}\label{prop:ECBV_is_trivialized_CBV}
  An exact commutative $BV_\infty$-algebra (\Cref{def:exact_commutative_BV_infty_algebra}) is also a trivialized commutative $BV_\infty$-algebra (\Cref{def:CBV/D_algebra}) with $\Delta_n=[i_n,d]$ and $\phi_n=i_n$ for $n\geq 1$.
\end{proposition}

\begin{proof}
  The operator $\Delta_n=[i_n,d]$ is of degree $1-2n$ and order $\leq n+1$ for the same reasons as in the proof of \Cref{lem:EBV_is_BV}. Applying the differential $[-,d]$ to \Cref{def:exact_commutative_BV_infty_algebra} \eqref{it_ECB_infty_2} we obtain 
  \begin{align*}
    0 & =\sum_{j=1}^{n-1}[[i_j,d],[i_{n-j},d]]=\sum_{j=1}^{n-1}[\Delta_j,\Delta_{n-j}]                                  \\
      & =\sum_{j=1}^{n-1}\left(\Delta_j\Delta_{n-j}+\Delta_{n-j}\Delta_j\right)=2\sum_{j=1}^{n-1}\Delta_j\Delta_{n-j}.
  \end{align*}
  Since \[[d,\Delta_n]=[d,[i_n,d]]=0,\] we conclude that the properties required in \Cref{def:commutative_BV_infty_algebra} are satisfied, so the operators $\Delta_n$, $n\geq1$, define a commutative $BV_\infty$-algebra structure.
  
  Moreover, by \Cref{def:exact_commutative_BV_infty_algebra} \eqref{it_ECB_infty_2}, in the following first summation, the factors $p\geq 2$ vanish, so
  \begin{align*}
    \sum_{p=1}^n\frac{1}{p!}\sum_{\substack{q_1+\cdots+q_p=n \\q_1,\dots,q_p\geq 1}}
    [i_{q_1},\dots[i_{q_{p-1}},[i_{q_p},d]]\dots]=[i_n,d]=\Delta_n.
  \end{align*}
  This is \eqref{eq:phi_CBV/D}.
  
  Notice that we have not used \Cref{def:exact_commutative_BV_infty_algebra} \eqref{it_ECB_infty_1} anywhere in this proof. Hence, this result would also hold if we excluded \eqref{it_ECB_infty_1} from \Cref{def:exact_commutative_BV_infty_algebra}.
\end{proof}

\begin{corollary}\label{cor:supermanifold_CBV_from_ECBV}
  The commutative $BV_\infty$-algebra structure on the de Rham complex $\CdR{M}$ of a graded manifold or supermanifold $M$ equipped with a generalized Poisson structure of the form $P=P_1+P_2+\cdots$ (\Cref{graded_manifold_CBV}) is induced by its exact commutative $BV_\infty$-algebra structure (\Cref{prop:supermanifold_ECBV}) and \Cref{prop:ECBV_is_trivialized_CBV}.
\end{corollary}

We can now derive one of the main results of \cite{braun_lazarev_2013_homotopy_bv_algebras}.

\begin{corollary}[{\cite[Theorem 4.4]{braun_lazarev_2013_homotopy_bv_algebras}}]
  The commutative $BV_\infty$-algebra structure on the de Rham complex $\CdR{M}$ of a graded manifold or supermanifold $M$ equipped with a generalized Poisson structure of the form $P=P_1+P_2+\cdots$ (\Cref{graded_manifold_CBV}) satisfies the degeneration property.
\end{corollary}

This follows from \Cref{prop:ECBV_is_trivialized_CBV} and \Cref{cor:supermanifold_CBV_from_ECBV}, since the degeneration property for commutative $BV_\infty$-algebras (\cite[Definition 3.13 and Remark 3.14]{braun_lazarev_2013_homotopy_bv_algebras} and \cite[Proposition 1.6]{dotsenko_shadrin_vallette_2015_rham_cohomology_homotopy}) is equivalent to the existence of a trivialization, see \cite[Theorem 2.1]{dotsenko_shadrin_vallette_2015_rham_cohomology_homotopy}. As a consequence of this result (\cite[Theorem 4.14 (1)]{katzarkov_kontsevich_pantev_2008_hodge_theoretic_aspects}), the underlying shifted $\Lie_\infty$-algebra structure is formal.

\begin{definition}\label{def:exact_CBV_operad}
  The \emph{exact commutative $BV_\infty$ operad} $\ECBV$ is the DG operad generated by $\mu\in\ECBV(2)^0$ and $i_n\in\ECBV(1)^{-2n}$, $n\geq 1$, with
  \[\mu\cdot (1\ 2)=\mu\]
  subject to the relations:
  \begin{enumerate}
    \item $d(\mu)=0$.
    \item $\mu\circ_1\mu-\mu\circ_2\mu=0$.
          \item\label{it:i_commutators} $[i_n,i_m]=0$, $n,m\geq 1$.
          \item\label{it:i_d(i)_commutators} $\sum_{j=1}^{n-1}[i_j,d(i_{n-j})]=0$, $n\geq 2$,
          \item\label{it:i_n_diff_op} $\sum_{\substack{p+q=n+1\\p\geq1}}(-1)^p\sum_{\sigma\in\shuffle{p}{q}}(\mu^{q+1}\circ_1(i_n\mu^{p}))\cdot\sigma^{-1}=0$, $n\geq 1$.
  \end{enumerate}
  Here \eqref{it:i_commutators} and \eqref{it:i_d(i)_commutators} correspond to \Cref{def:exact_commutative_BV_infty_algebra} \eqref{it_ECB_infty_1} and \eqref{it_ECB_infty_2}, respectively, and \eqref{it:i_n_diff_op} is the equation in \Cref{def:differential_operator}. The suboperad generated by $\mu$ is the commutative operad $\C\subset\ECBV$.
\end{definition}

The operadic counterparts of \Cref{prop:ECBV_is_trivialized_CBV} and \Cref{cor:supermanifold_CBV_from_ECBV} are the following two results.

\begin{proposition}\label{prop:CBV/D_to_ECBV}
  There is a morphism of DG operads $\CBV/\Delta\to\ECBV$ defined by $\mu\mapsto\mu$, $\Delta_n\mapsto d(i_n)$, and $\phi_n\mapsto i_n$, $n\geq 1$.
\end{proposition}

\begin{proposition}\label{prop:generalized_poisson_factorization}
  Given a graded manifold or supermanifold $M$ equipped with a generalized Poisson structure of the form $P=P_1+P_2+\cdots$, the morphism $\CBV\to\E{\CdR{M}}$ determined by \Cref{graded_manifold_CBV} factors as \[\CBV\longrightarrow\CBV/\Delta\longrightarrow\ECBV\longrightarrow\E{\CdR{M}},\]
  where the first two arrows are the ones in \Cref{def:trivialized_CBV_operad} and \Cref{prop:CBV/D_to_ECBV}, and the third arrow is given by \Cref{prop:supermanifold_ECBV}.
\end{proposition}

We now proceed with the main \namecref{thm:ECBV} of this section.

\begin{theorem}\label{thm:ECBV}
  The DG operad inclusion $\C\subset\ECBV$ is a quasi-isomorphism.
\end{theorem}

\begin{proof}
  We consider the cochain complex $C=\bigoplus_{n\geq 1}C_n$, where
  \begin{center}
    \begin{tikzpicture}[xscale=1.3]
      \node at (-1,0) {$C_n\colon$};
      \node (C0) at (0,0) {$\cdots$};
      \node (C1) at (1,0) {$0$};
      \node (C2) at (2,0) {$k$};
      \node (C3) at (3,0) {$k$};
      \node (C4) at (4,0) {$0$};
      \node (C5) at (5,0) {$\cdots$};
      \draw[->] (C0) -- (C1);
      \draw[->] (C1) -- (C2);
      \draw[->] (C2) -- (C3);
      \draw[->] (C3) -- (C4);
      \draw[->] (C4) -- (C5);
      \node at (-.1,.5) {\scriptsize degree};
      \node at (.9,.5) {$\scriptstyle -1-2n$};
      \node at (1.9,.5) {$\scriptstyle -2n$};
      \node at (3,.5) {$\scriptstyle 1-2n$};
      \node at (4,.5) {$\scriptstyle 0$};
      \node at (5,.5) {$\cdots$};
      \node at (-.3,-.5) {\scriptsize generators};
      \node (i) at (2,-.5) {$i_n$};
      \node (di) at (3,-.505) {$d(i_n)$};
      \draw[|->] (i) -- (di);
    \end{tikzpicture}
  \end{center}
  We have an SDR
  \begin{center}
    \SDR{0}{C}[][][h]
  \end{center}
  given by
  \[h(i_n)=0,\qquad hd(i_n)=i_n,\qquad n\geq1.\]
  We consider $C$ as an $\sym$-module concentrated in arity $1$, so the previous SDR is an $\sym$-module SDR. The operad $\ECBV$ is the quotient of $\P=\C\amalg\F(C)$ by the DG ideal $I$ generated by 
  \begin{enumerate}
    \item $[i_n,i_m]$, $n,m\geq 1$.
          \item\label{it:last_proof_2} $\sum_{j=1}^{n-1}[i_j,d(i_{n-j})]$, $n\geq 2$,
    \item $\sum_{\substack{p+q=n+1\\p\geq1}}(-1)^p\sum_{\sigma\in\shuffle{p}{q}}(\mu^{q+1}\circ_1(i_n\mu^{p}))\cdot\sigma^{-1}$, $n\geq 1$.
  \end{enumerate}
  These generators are weight-homogeneous. The first two items have weight $2$, and the third one has weight $1$. The homotopy $h_C$ of \Cref{thm:operad_SDR} maps the first and the third items to $0\in I$. 
  Indeed, the claim for the third item is easy because it has weight $1$ and $h(i_n)=0$. For the first item, by \Cref{thm:operad_SDR} \eqref{it:h_composition},
  \begin{align*}
    h([i_n,i_m]) & =h(i_ni_m-i_mi_n)                                                                    \\
                 & =\frac{1}{2}h(i_n)i_m+\frac{1}{2}i_nh(i_m)-\frac{1}{2}h(i_m)i_n-\frac{1}{2}i_mh(i_n) \\
                 & =0+0-0+0                                                                             \\
                 & =0.
  \end{align*}
  As for the second item, again by \Cref{thm:operad_SDR} \eqref{it:h_composition},
  \begin{align*}
    h([i_j,d(i_{n-j})]) & = h(i_jd(i_{n-j}))-h(d(i_{n-j})i_j)                                                                        \\
                        & = \frac{1}{2}h(i_j)d(i_{n-j})+\frac{1}{2}i_jhd(i_{n-j})-\frac{1}{2}hd(i_{n-j})i_j+\frac{1}{2}i_{n-j}h(i_j) \\
                        & =0+\frac{1}{2}i_ji_{n-j}-\frac{1}{2}i_{n-j}i_j+0                                                           \\
                        & =\frac{1}{2}[i_j,i_{n-j}]\in I,
  \end{align*}
  so $h\eqref{it:last_proof_2}\in I$ as well, and \Cref{thm:operad_SDR} applies. Notice that here we use \Cref{def:exact_commutative_BV_infty_algebra} \eqref{it_ECB_infty_1}, which we did not need in the proof of \Cref{prop:ECBV_is_trivialized_CBV}.
\end{proof}

We now consider triviality and formality results. 

The following corollary can be applied to $\P=\CBV$ the commutative $BV_\infty$-operad, see \Cref{prop:generalized_poisson_factorization}, and also to  $\P=\Lie_\infty[1]$ the shifted $\Lie_\infty$-operad since there is an operad morphism
$\Lie_\infty[1]\to\CBV$, see \cite{kravchenko_2000_deformations_batalin_vilkovisky} or \cite[Corollary 3.3]{braun_lazarev_2013_homotopy_bv_algebras}.

\begin{corollary}
  Given a graded manifold or supermanifold $M$ equipped with a generalized Poisson structure of the form $P=P_1+P_2+\cdots$, a graded operad $\P$ concentrated in degrees $\leq 0$, and an operad morphism $\P\to\ECBV$, the $\P$-algebra structure on $\CdR{M}$ obtained from \Cref{prop:supermanifold_ECBV} by restriction of scalars is quasi-isomorphic to the trivial $\P$-algebra structure.
\end{corollary}

The proof is like that of \Cref{cor:P-trivial}, \emph{mutatis mutandis}.

The following corollary is mainly interesting for $\P=\CBV$.

\begin{corollary}
  Let $M$ be a graded manifold or supermanifold equipped with a generalized Poisson structure of the form $P=P_1+P_2+\cdots$, $\P$ a graded operad concentrated in degrees $\leq 0$, $\P^0\subset\P$ its degree $0$ suboperad, and $\P\to\ECBV$ an operad morphism. We use it together with \Cref{prop:supermanifold_ECBV} to equip $\CdR{M}$ with a $\P$-algebra structure. It has an underlying $\P^0$-algebra structure obtained by restriction of scalars along the inclusion of the degree $0$ part $\P^0\subset\P$. The following statements hold:
  \begin{enumerate}
    \item If $M$ is formal then it is also $\P^0$-formal.
    \item $M$ is $\P$-formal if and only if it is $\P^0$-formal.
  \end{enumerate}
\end{corollary}

\begin{remark}
  The exact Batalin-Vilkovisky operad and the exact commutative $BV_\infty$-operad are related. The former is a quotient of the latter, since we have a morphism $\ECBV\to\EBV$ defined by $\mu\mapsto\mu$, $i_1\mapsto i$, and $i_n\mapsto 0$ for $n\geq 2$. This map is a quasi-isomorphism by \Cref{thm:Poisson,thm:ECBV}. 
  
  We could have deduced most results in \Cref{sec:EBV} from \Cref{thm:ECBV}, without \Cref{thm:Poisson}. However, \Cref{thm:Poisson} is valuable in itself because it shows the impossibility of defining meaningful higher algebraic structures on the de Rham complex of a Poisson manifold out of the exterior product and the interior product with the bivector field.
\end{remark}

\printbibliography
\end{document}